%% file: hochdpsapde.tex
\documentclass[12pt,a4paper]{article}

\title{Learning high-order spatial discretisations of PDEs with symmetry-preserving iterative algorithms} 

\author{J.~E. Bunder\thanks{\url{mailto:judith.bunder@adelaide.edu.au}, {\sc orcid:}0000-0001-5355-2288} \quad and \quad A.~J. Roberts\thanks{\url{mailto:profajroberts@protonmail.com}, {\sc orcid:}0000-0001-8930-1552 }
\\School of Mathematical Sciences,
\\University of Adelaide, South Australia.
}

\date{\today}

\AtBeginDocument{\RaisedNamesfalse}
\usepackage[hypertexnames=false]{hyperref} 
\IfFileExists{ajr.sty}{\usepackage{ajr}}{}
\usepackage{listings}
\usepackage{reducecode}
\usepackage{amsmath,amsfonts,mathrsfs}
\usepackage{enumitem,microtype}

\usepackage{pgfplots}\pgfplotsset{compat=newest}
\usepgfplotslibrary{external}

\usepackage{eulervm,defns,mybiblatex,mycleveref}
\let\LTXappendix\appendix
\def\appendix{\printbibliography%
    \gdef\printbibliography{}%
    \clearpage\LTXappendix}
    
\allowdisplaybreaks

\renewcommand{\vec}[1]{\text{\boldmath\(#1\)}}

\Cal B\Cal C\Cal D\Cal E\Cal K\Cal L
\Cal R\Cal T\Cal U
\Bb C\Bb E\Bb H\Bb I\Bb J\Bb L\Bb R\Bb U\Bb X
\Vec f\Vec u\Vec x\Vec z \Vec c
\Vec G\Vec L\Vec U
\Frak f \Frak L
\def\asinh{\operatorname{asinh}}

\def\fx{{\cal%
    \mathchoice{\scriptstyle X}{\scriptstyle X}%
    {\scriptscriptstyle X}{\scriptscriptstyle X}}}
\def\fu{{\cal%
    \mathchoice{\scriptstyle U}{\scriptstyle U}%
    {\scriptscriptstyle U}{\scriptscriptstyle U}}}

\def\tu{\ensuremath{\tilde u}}

\def\tf{\ensuremath{\tilde f}}

\def\pded{\cref{eqpdesd}}
\def\rhs{\textsc{rhs}}
\def\lhs{\textsc{lhs}}
\def\cmt{\textsc{cmt}}

\begin{document}

\maketitle



\begin{abstract}
Common techniques for the spatial discretisation of \pde{}s on a macroscale grid include finite difference, finite elements and finite volume methods. 
Such methods typically impose assumed microscale structures on the subgrid fields, so without further tailored analysis are not suitable for systems with subgrid-scale heterogeneity or nonlinearities. 
We provide a new algebraic route to systematically approximate, in principle exactly, the macroscale closure of the spatially-discrete dynamics of a general class of heterogeneous non-autonomous reaction-advection-diffusion \pde{}s.
This holistic discretisation approach, developed through rigorous theory and verified with computer algebra, systematically constructs discrete macroscale models through physics informed by the \pde\ out-of-equilibrium dynamics, thus relaxing many assumptions regarding the subgrid structure. 
The construction is analogous to recent gray-box machine learning techniques in that predictions are directed by iterative layers (as in neural networks), but informed by the subgrid physics (or `data') as expressed in the \pde{}s. 
A major development of the holistic methodology, presented herein, is novel inter-element coupling between subgrid fields which preserve self-adjointness of the \pde{} after macroscale discretisation, thereby maintaining the spectral structure of the original system.
This holistic methodology also encompasses homogenisation of microscale heterogeneous systems, as shown here with the canonical examples of heterogeneous 1D waves and diffusion.


\end{abstract}

\tableofcontents

\section{Introduction}

The observable macroscale dynamics of a multiscale system are typically driven by coherent structures governed by microscopic agents, such as electrons, molecules, or individuals in a population.
While such systems are often well-understood at the microscale, large-scale simulations of the full microscale model are usually infeasible, even on the largest high-performance computers, and in particular for nonlinear system with fluctuations that persist and interact across multiple scales \cite[]{Nasa2018}.
Multiscale modelling aims to overcome this issue by efficiently approximating over the multiple scales to provide an accurate numerical model of the coherent dynamics of the system at the macroscale of interest.
Numerous computational schemes have been developed for a wide variety of multiscale systems, from composite materials~\cite[]{Dutra2020} to  plasmas~\cite[]{Shohet2020}. 
However, many of these computational schemes are developed for specific systems and are not easily adaptable to other cases.
Furthermore, supporting theory is often only rigorously applicable in the limiting case of infinite scale separation (typically characterised as \(\epsilon\rightarrow 0\)).  
Here we further develop the holistic discretisation methodology for multiscale nonlinear systems---a general purpose approach---supported by analytical theory and straightforwardly adaptable to a wide range of spatio-temporal systems at finite scale separation.


To determine a macroscale description of a multiscale system with microscale heterogeneity we seek an homogenisation, or `average', over the microscale which accounts for the heterogeneity without retaining unnecessary fine-scale details~\cite[]{Geers2107}. 
For example, for composite materials with periodic microstructures, which are common in nature and are increasingly synthesised for novel industrial applications \cite[e.g.,][]{NematNasser2011,Bargmann2018},  asymptotic homogenisation constructs a power series in scale separation~\(\epsilon\) with coefficients  dependent on a periodic `representative volume element' (e.g., a unit cell) which describes the microscale heterogeneity \cite[]{Dutra2020,Feppon2020}.
However, homogenisation often requires substantial user input in the form of an algebraic analysis of the given microscale system prior to numerical implementation, and theoretical support is often only guaranteed in unphysical circumstances; for example, typically one needs to explicitly identify `fast' and `slow' variables, and also assume an unphysical infinite scale separation between the two  \cite[e.g.,][]{Engquist08}.
In contrast, the holistic homogenisation methodology, detailed in \cref{seccte}, requires little or no such user input and generally permits an immediate application of the numerical scheme.
It is thus more adaptable and accessible than many other common homogenisation methods. 
Theoretical support for this methodology is provided by centre manifold theory, as discussed in \cref{seclin},  and by the consistency analysis of \cref{secphoc1}, and is broadly applicable to many physical scenarios. 

The holistic homogenisation methodology was originally developed by \citet[e.g.,][]{Roberts98a}, and has been applied to several different systems \cite[e.g.,][]{Roberts00a, MacKenzie09b, Roberts2011a, Jarrad2016a}.
The fundamental idea is to partition the domain of the original system into \(N\)~disjoint elements, with these \(N\)~elements defining a macroscale grid, and then to discover appropriate microscale fields that are slaved to a useful set of macroscale variables. 
In this article, the important new development is the preservation of self-adjointness by appropriate design of the coupling between the \(N\)~elements, as defined in \cref{seccte}, which ensures that fundamental dynamical features of the multiscale system (specified by the eigenstructure) are maintained---especially for homogenisation as proved in \cref{secpsa}.  
To illustrate the adaptability of this holistic methodology, \cref{sechhd} considers a \pde\ with periodic heterogeneous diffusion.
The problem is firstly embedded in a family of general heterogeneous diffusion \pde{}s, which rephrases the heterogeneity as solely dependent on a `phase' coordinate, where the periodic phase domain is reminiscent of the representative volume element commonly utilised in the asymptotic homogenisation of periodic microstructures.
This embedding empowers a straightforward application of the numerical holistic methodology.


In this article, to establish the new homogenisation methodology we mostly consider the spatio-temporal dynamics of a field~\(u(t,x)\) satisfying non-autonomous reaction-advection-diffusion \pde{}s in the general form
\begin{equation}
u_t=-f(x,u,u_x)_x+\alpha g(t,x,u,u_x)\,,
\label{eqgenpde}
\end{equation}
where subscripts \(x\) and~\(t\) denote spatial and temporal derivatives, respectively, and functions \(f\)~and~\(g\) are characterised by~\cref{assFandG}.
The homogenisation derives accurate macroscale spatial discretisations of these nonlinear \pde{}s on a specified macroscale grid.
The homogenisation accounts for subgrid (i.e., microscale) dynamics and symmetries, and hence automatically leads to stable discretisations, assuming stability of the original system.
We also establish that the same approach generates accurate and stable spatial discretisations of 1D first-order wave \pde{}s (\cref{sechdudwpde}) and 1D second-order nonlinear wave \pde{}s obtained by replacing~\(u_t\) in~\cref{eqgenpde} by~\(u_{tt}\) (\cref{secsmwpde}).

With some adaptations, this methodology is also suitable to multi-dimensional space \pde{}s, \(u_t=-\divv\fv(\xv,u,\grad u)+\alpha g(t,\xv,u,\grad u)\)\,, and is related to the multidimensional homogenisation (albeit non-self-adjoint) of  \cite{Roberts2011a}, but we leave multi-dimensional space for future research and here we focus the case \text{of 1D space.}

\begin{assumption}\label{assFandG}
Firstly, functions~\(f\) and~\(g\) on the right-hand side of \pde{}~\cref{eqgenpde} are to be smooth functions of their arguments, and the function~\(g\) varies relatively slowly in time~\(t\).
Secondly, in the cases of second-order \pde{}s, the function~\(f\), called the \emph{flux}, is strictly monotonic decreasing in~\(u_x\): that is, for some positive~\(\nu\),
\(\D{u_x}f\leq-\nu<0\) with \(f(x,u,0)=0\)\,.
Lastly, for strict support by existing theory, we assume \(f\)~and~\(g\) are such that the \pde~\cref{eqgenpde}, with `edge conditions'~\cref{eqsiecc}, satisfies the requirements of the invariant manifold theory by \cite{Hochs2019}.
\end{assumption} 

We define \emph{smooth} (\cref{assFandG}) to mean either infinitely differentiable,~\(C^\infty\), or differentiable to an order~\(p\) sufficient for the purposes at hand,~\(C^p\).
This smoothness requirement still permits microscale heterogeneity where the functions~\(f\) and~\(g\) vary significantly on a spatial scale of the order of a microscale length~\(d\), such as for heterogeneous diffusion (\cref{sechhd}).

We denote the 1D spatial domain of the \pde~\cref{eqgenpde} by~\XX, which in examples is often over the interval \((0,\fL)\subset\RR\)\,, and consider spatial dependence of field~\(u\) in the Hilbert space~\HH\ of square integrable, twice differentiable, functions on~\XX, and often in the subspace~\LL\ of all \(\fL\)-periodic functions in~\HH\ (i.e., we often invoke periodic boundary conditions). 
We partition the domain~\XX\ into \(N\)~disjoint open interval elements~\(\XX_j\) for \(j\in\JJ:=\{1,\ldots,N\}\) and define \(\tilde\XX:=\cup_{j\in\JJ} \XX_{j}\)\,. 
Let~\(u_j(t,x)\) denote subgrid fields which satisfy the \pde~\cref{eqgenpde} in~\(\XX_j\): that is,
\(u_{jt}=-f(x,u_j,u_{jx})_x+\alpha g(t,x,u_j,u_{jx})\),
for \(x\in\XX_j\)\,.
Define macroscale, coarse, `grid' values 
\begin{equation}
U_j(t):=\cU\big[ u_j(t,x)\big] \quad\text{for every } j\in\JJ,
\label{eqgridval}
\end{equation}
for some chosen projection \(\cU: C(\XX_j)\to\RR\)\,, a projection that is some measure of the amplitude of the field~\(u_j\) in the \(j\)th~element.\footnote{The precise form of~\cU\ is largely immaterial because the evolution of states in the slow subspace is independent of how we choose to parametrise the subspace \cite[Lemma~5.1]{Roberts2014a}.  Hence the macroscale projection~\cU\ may be subjectively chosen according to any reasonable measure of the subgrid field~\(u_j\) within each element~\(\XX_j\).}
Consequently, a spatial discretisation of the \pde~\cref{eqgenpde} is a \emph{closed} set of \ode{}s for the vector of macroscale grid variables \(\Uv:=(U_1,\ldots,U_N)\) in the form 
\begin{equation}
\de t\Uv=\Gv(t,\Uv).
\label{eqcsd}
\end{equation}
Four examples discussed herein are the \ode\ systems~\cref{eq:schromacro,eqgopdiff,eqgopwave1,eqgophetdiff} modelling, respectively, nonlinear wave modulation, spatial diffusion, progressive waves, and homogenised heterogeneous diffusion.
\cref{seclin} uses centre manifold theory (\cmt) \cite[e.g.,][]{Carr81, Haragus2011, Hochs2019} to do three things in support of such a macroscale evolution~\cref{eqcsd}. 
Firstly, we develop further an approach to proving that in principle there exists an \emph{exact} macroscale closure~\cref{eqcsd} to the dynamics of the \pde~\cref{eqgenpde}. 
Secondly, \cmt\ establishes that such a closure is emergent from general initial conditions. 
Thirdly, \cmt\ justifies constructing new systematic approximations to the in-principle closure by learning subgrid field corrections from the \pde~\cref{eqgenpde}.

Traditional spatial discretisations of \pde{}s~\cref{eqgenpde}, whether finite difference, finite element, or finite volume, impose a set of assumed subgrid fields within each element and then derive approximate rules for the macroscale evolution of the parameters of the imposed fit \cite[e.g.,][]{Efendiev2004, Geers2010}.
The accuracy of this modelling by discretisation, particularly for nonlinear systems, is strongly dependent on the class of imposed subgrid microstructure and its ability to capture the microscale dynamics which emerge and persist across the macroscale \cite[e.g.,][]{Aarnes2007, Zhang2020}. 
Our dynamical systems (holistic) approach is to \emph{learn the algebraic} subgrid fields from the \pde~\cref{eqgenpde}.
Using more general subgrid structures is cognate to the aims of the so-called Generalised\slash Extended Finite Element Methods \cite[e.g.,][]{Turner2011b}, but crucially we let the \pde{} determine the subgrid structures from the powerful framework of centre manifolds in dynamical systems.
The so-called stabilized scheme \cite[e.g.,][]{Hughes95} appears analogous to the first step of our construction in \cref{seclin}.
We learn improved subgrid fields via computer algebra using the governing \pde~\cref{eqgenpde}, with second and further iterations providing higher-order terms.  
The resultant holistic discretisation is both accurate and flexible---it naturally captures fine-scale structures which persist across multiple scales, and it does so to a chosen order of accuracy that is routinely achievable via iteration.

A previous alternative holistic approach learns subgrid fields by systematically refining a piecewise constant initial approximation \cite[e.g.,][]{Roberts98a, Roberts00a, Roberts2011a}---an approach that adapts to the multi-scale gap-tooth scheme \cite[e.g.,][]{Roberts06d, Kevrekidis09a}.
Our new approach extends and complements previous schemes by systematically preserving self-adjoint symmetries in the learnt macroscale closure.

\begin{example}[\textsc{nls}]\label{ex:nls}
As an introductory nontrivial example, consider the nonlinear 1D Schr\"odinger equation (\textsc{nls}) governing a complex valued field~\(u(t,x)\):
\begin{equation}
\i\partial_t u=-\tfrac{1}{2}\partial_x^2 u+\alpha|u|^2u\,,\label{eq:schrodinger}
\end{equation}
with periodic boundary conditions on the spatial domain~\([-\pi,\pi]\).
The real and imaginary parts of the field~\(u\) typically evolve to be out-of-phase and each drives persistent oscillations in the other.  
Common applications include the modulation of light propagating in nonlinear optical fibres and planar waveguides~\cite[e.g.]{Kibler2012}.
The nonlinear wave nature of this example tests our discretisation scheme.

We discretise space into \(N\)~elements~\(\XX_j\) of uniform width~\(H\) centred about grid points~\(X_j\).
Then, defining macroscale complex variables \(U_{j}(t):=u_j(t,X_j)\) (i.e., here the projection~\(\cU\) of~\cref{eqgridval} is the subgrid field~\(u_j\) at the midpoint~\(X_j\) of the \(j\)th~element), our holistic homogenisation scheme iteratively derives a macroscale discretisation of~\cref{eq:schrodinger}.
The details of the algebraic machinations of the iteration are not important here---computer algebra like that of \cref{sechochd,sechochwave1,sechdhdc} handles the details.
The result here is that the macroscale discretisation is, with asterisk superscript denoting complex conjugation,
\begin{align}
\i\partial_t U_j&=\alpha|U_j|^2U_j 
-\gamma^2\tfrac{1}{8H^2}(U_{j+2}+U_{j-2}-2U_j)
\nonumber\\&\quad{}
+\alpha\gamma^2\tfrac{1}{96}\big[2|U_j|^2(U_{j+2}+U_{j-2}-4U_j)
\nonumber\\&\qquad{}
-4|U_{j+2}|^2(U_{j+2}-2U_j)-4|U_{j-2}|^2(U_{j-2}-2U_j)
\nonumber\\&\qquad{}
-4U_j^*(U_{j+1}^2+U_{j-1}^2-U_{j+1}U_{j-1})-4U_j(U_{j+1}^*U_{j-1}+U_{j-1}^*U_{j+1})
\nonumber\\&\qquad{}
+U_{j+2}^*(U_j^2-2U_{j+1}^2)+U_{j-2}^*(U_j^2-2U_{j-1}^2)\big]
+\Ord{\alpha^2,\gamma^3},
\label{eq:schromacro}
\end{align}
where \(\gamma\) is a coupling order parameter labelling each inter-element interaction (\cref{seccte}). 
Set \(\gamma=1\) to obtain the full inter-element coupling needed to model the \pde{}, and then the first line of the macroscale evolution~\cref{eq:schromacro} is a standard discretisation of the \textsc{nls}~\cref{eq:schrodinger} with the spatial derivatives over a step of~\(2H\).
The subsequent terms in~\cref{eq:schromacro} are the leading terms due to subgrid physical interactions, terms learnt by the algebraic methodology developed \text{in this article.}

\begin{figure}
\includegraphics[width=\textwidth]{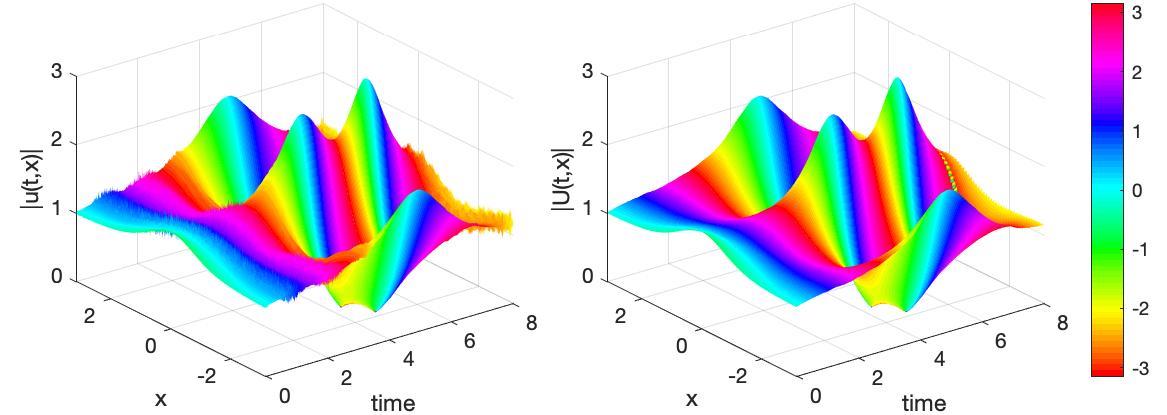} 
\caption{Simulations of the \textsc{nls} \pde\ compare predictions of field magnitudes: (left)~\(|u(t,x)|\) computed using a fine-scale discretisation of~\cref{eq:schrodinger}; and (right)~\(|U(t,x)|\) computed using the coarse scale holistic discretisation~\cref{eq:schromacro}.
The surfaces are coloured according to the arguments of the complex fields, as indicated by the side-bar. 
}
\label{fig:schrosim} 
\end{figure}%
For an example simulation we choose the initial condition \(u(0,x)=1-\i\operatorname{sech}x\) which produces Kuznetsov--Ma breather solutions of nonlinear localised oscillating peaks on a non-zero background \citep{Zhao2017}.
Breathers and solitons arise in nonlinear \pde\ when there is a balance between the dispersion and the nonlinearity.
\cref{fig:schrosim} illustrates the success of the holistic homogenisation by comparing a microscale simulation of the \textsc{nls}~\cref{eq:schrodinger} (left plot) with the evolution of the macroscale discretisation~\cref{eq:schromacro} truncated to errors~\(\Ord{\alpha^2,\gamma^4}\) (right).
The microscale simulation of the \textsc{nls} employs a spatial grid of \(3300\)~points with spacing \(h=0.0019\), whereas the holistic scheme discretises the domain with \(N=151\) elements of width \(H=0.042\)\,, which is over twenty times larger and correspondingly less stiff, and so much quicker computationally.
At all times we observe fine-scale simulation errors (via Matlab's \textsc{ode} solver \texttt{ode113}) which do not appear in the simulation with the holistic discretisation~\cref{eq:schromacro} (via a fourth order Runge--Kutta).\footnote{Matlab's \texttt{ode113} was the only Matlab \textsc{ode} solver which could simulate the fine-scale model~\eqref{eq:schrodinger} and provide reasonable accuracy (time taken: {8\,564\,s}). 
In contrast, Matlab \textsc{ode} solvers had no trouble accurately simulating the holistic discretisation~\eqref{eq:schromacro} (e.g., \texttt{ode113} time was~{246\,s}), but fourth order Runge--Kutta was substantially faster (time~{102\,s}).
The fourth order Runge--Kutta could not accurately simulate the  \text{fine-scale model.}} 
\end{example}

\begin{remark}\label{rem:iecc}
\emph{Inter-element coupling is crucial.}\quad
In order to construct an accurate macroscale discretisation the nature of the coupling between adjacent elements~\(\XX_j\) is crucial.
For example, in the context of material homogenisation, \cite{Abdulle2020b} in their abstract comment that
\begin{quote} 
A naive treatment of these artificial boundary conditions leads to a first order error \ldots\,. This error dominates all other errors originating from the discretization of the macro and the micro problems, and its reduction is a main issue in today's engineering multiscale computations. 
\end{quote}
The ``naive treatment'' in this comment corresponds to a low-order approximation in the element coupling order parameter~\(\gamma\), but with our iterative algorithm higher order approximations are straightforward.
The problematic ``boundary conditions'' in the comment we term the \emph{edge conditions} (or coupling conditions)---it is these edge conditions that define the coupling of elements~\(\XX_j\) to form the spatially complete system.
\cref{seccte} develops inter-element edge conditions that in a variety of scenarios (e.g., \textsc{nls} of \cref{ex:nls}) provably preserve self-adjointness (\cref{secpsa}), and have provable accuracy (\cref{sechdudwpde,Xsecpchdd}). 
Thus this article develops a valuable approach to addressing a key ``issue in today's engineering multiscale computations''.
\end{remark}

\begin{remark}\label{rem:ic}
\emph{Initial conditions.}\quad
For any macroscale discrete closure such as~\cref{eqcsd}, initial conditions are nontrivial.
The paradox is that, despite the definition~\cref{eqgridval} that the macroscale grid variable \(U_j(t):=\cU [u_j(t,x)]\), in order to obtain accurate forecasts over long times, generally the initial grid value \(U_j(0)\neq \cU [u_j(0,x)]\) \cite[]{Roberts89b,Roberts01a}.
In physics, \cite{vanKampen85} termed  this phenomena the ``initial slip''. 
There is, in effect, a `boundary layer' in time that accounts for non-trivial rapid transients.
Using the geometry of invariant manifolds, \cite{Roberts89b} developed an efficient general method to algebraically learn the correct initial~\(U_j(0)\) for accurate long-time forecasts by reduced dimensional models such as a discretisation of a \pde.
The method has been applied in various scenarios \cite[e.g.,][]{Roberts2014a}.
Further research is needed to form initial conditions for the discretisation framework \text{established herein.}
\end{remark}

\begin{remark}\label{rem:awml}
\emph{Analogy with machine learning.}\quad  
The iterative algorithm used to construct the holistic homogenisation of \cref{ex:nls}, and the other examples herein, was originally developed by \cite{Roberts96a}. 
At each iteration the algorithm evaluates residuals of the nonlinear~\textsc{pde} within the elements, and improves the resolution of the subgrid fields with a linear correction based upon this subgrid information.
We draw an analogy between this iteration and a machine learning algorithm where an AI learns the generic form of the macroscale evolution from many thousands of simulations \cite[e.g.,][]{GonzalezGarcia98, Frank2020,  Linot2020}, but here the AI is a `smart' grey-box which is directed by additional algebraic knowledge\slash data.
For example, \cite{BarSinai2018} imposed a conservative form on their learning of a  coarse-scale closure to Burgers' \pde, and \cite{Mercer94a} constructed stable schemes to advection-diffusion in pipes that matched the long term evolution to high order. 
In constructing the holistic discretisation, each iteration is analogous to one layer in a deep neural network: evaluating residuals is analogous to a nonlinear neurone function; and the linear corrections are analogous to using weighted linear combinations of outputs (i.e., activation functions) of one layer as the inputs of the next layer.  
The analogy requires that, at each execution, the smart neural network is directed by the holistic algorithm to specifically craft the neurones and linear weights to the problem at hand.
Being algebraic, this data which directs the algorithm encompasses all points in the state space's domain, not just sample data as in machine learning.
Consequently, we contend that mathematicians have for many decades been doing smart analogues of machine learning.
Such algebraic learning empowers the physical interpretation, validation, and verification required by modern science \cite[e.g.][]{Brenner2021}.
\end{remark}

\section{Self-adjoint preserving coupling}
\label{secsapc}

Let the 1D spatial discretisation of the general \pde~\eqref{eqgenpde} be parametrised by \(\Uv=(U_1,\ldots,U_N)\) as defined by some projection~\eqref{eqgridval} and satisfy an evolution \ode\ of the form~\cref{eqcsd}. 
The discretisation is derived from the subgrid fields~\(u_j(t,x)\) of \pde{}~\eqref{eqgenpde} over the \(N\)~disjoint elements~\(\XX_j\). 
To construct the spatial discretisation and \ode~\cref{eqcsd} we must specify~\(u_j\) at the two edges of element~\(\XX_j\). 
Specifying the fields~\(u_j\) at the edges of element~\(\XX_j\) defines an inter-element coupling.
Such coupling conveys information across space and hence engenders the macroscale dynamics.
While there are many possible coupling conditions, here we establish new coupling conditions which both preserve self-adjoint symmetry in the original microscale \pde, and also ensure the macroscale dynamics are correctly consistent to an arbitrarily high order of accuracy. 
As these new coupling conditions constrain fields on the edges of the elements~\(\XX_j\), we refer to them as \emph{edge conditions}.

We identify the right and left edge-points of the elements by writing the \(j\)th~element as the open interval \(\XX_j=(L_j,R_j)\). 
Because the elements abut, the right-edge of the \(j\)th~element is the left-edge of the \((j+1)\)th~element, \(R_j=L_{j+1}\)\,, and thus, to recover the original \pde~\eqref{eqgenpde} over the entire domain, we aim to set \(u_j(t,R_j)=u_{j+1}(t,L_{j+1})\).%
\footnote{Strictly, such edge values are the limit as~\(x\) approaches the edge of the (open) elements.}
However, to understand and theoretically support the discretisation on the elements, we parameterise a range of inter-element coupling to best control the information flow between elements. 
We introduce an real-valued inter-element coupling parameter~\(\gamma\): 
when \(\gamma=0\) the elements are uncoupled and form a base for some theory; 
when \(\gamma=1\) the elements are fully coupled to recover the original \pde\ problem~\cref{eqgenpde} for field~\(u(t,x)\).
A solution field~\(u(t,x)\) then arises as the collection of subgrid fields~\(u_j(t,x)\).
Further, we define the order parameter~\(\gamma\) so as to `label' each inter-element communication so that a term in~\(\gamma^p\) expresses a composition of interactions between each element and \(p\)~pairs of its nearest neighbour elements.
Consequently, truncating asymptotic expansions to `errors'~\ord{\gamma^{p}} naturally and automatically creates local discretisations of stencil \text{width~\((2p+1)\) in space. }

\subsection{Couple the elements}
\label{seccte}
For conciseness, denote the flux in the \(j\)th~element as \(f_j(t,x):=f(x,u_j,u_{jx})\)\,.
Further, use superscripts~\(R,L\) to denote evaluation at the corresponding right\slash left edges \(x=R_j,L_j\)\,:
for example, \(u_j^R(t):=u_j(t,R_j)\) and  \(f_j^L(t):=f_j(t,L_j)\)\,.
In addition to the inter-element coupling parameter~\(\gamma\), we introduce a second real-valued parameter~\(\theta\) that flexibly `tunes' the details of inter-element communication; typically \(|\theta|\leq1\)\,.
For example, for advective processes, \(\theta\)~controls the upwind character of the discrete model.
At every time and every element~\(j\), we couple elements with the two edge conditions
\begin{subequations}\label{eqsiecc}%
\begin{align}
(1-\tfrac12\gamma)(u_j^R-u_j^L)&=\tfrac12\gamma(u_{j+1}^L-u_{j-1}^R)+\tfrac12\gamma\theta(u_j^L+u_j^R)-\tfrac12\gamma\theta(u_{j+1}^L+u_{j-1}^R)\,,
\label{eqiecca}
\\
(1-\tfrac12\gamma)(f_j^R-f_j^L)&=\tfrac12\gamma(f_{j+1}^L-f_{j-1}^R)-\tfrac12\gamma\theta(f_j^L+f_j^R)+\tfrac12\gamma\theta(f_{j+1}^L+f_{j-1}^R).
\label{eqieccb}
\end{align}
\end{subequations}
When \(u(t,x)\) is to be \fL-periodic in space we require 
the periodicity \(u^{e}_{j+N}=u^{e}_j\) and \(f^{e}_{j+N}=f^{e}_j\) for \(e\in\{L,R\}\)\,.
\cref{secpsa} proves that these inter-element edge conditions preserve self-adjointness in the \pde~\cref{eqgenpde}.

\begin{itemize}
\item When parameter \(\gamma=0\)\,, the edge conditions~\cref{eqsiecc} reduce to
\begin{align}
u_j^R-u_j^L=0\,,\quad f_j^R-f_j^L=0\,.
\label{eqsiecc0}
\end{align}
That is, each element is uncoupled from the others when \(\gamma=0\)\,.
Consequently we find that \(\gamma=0\) forms a useful base for theory.
Denoting the element lengths by \(H_j:=R_j-L_j\)\,, the two conditions~\eqref{eqsiecc0} specify that the subgrid field~\(u_j\) is \(H_j\)-periodic.
For the forthcoming \cref{eghocdd} on the diffusion \pde, \(u_t=u_{xx}\)\,, such uncoupled elements evolve in time so that \(u_j\to{}\)constant on the cross-element diffusion time scale of~\(1/H_j^2\).
Centre manifold theory then supports the existence and emergence of an exact closure to the spatial discretisation, namely~\eqref{eqcsd}, in some finite range of~\(\gamma\) about \(\gamma=0\) (\cref{seclin}).

\item When parameter \(\gamma=1\) the edge conditions~\cref{eqsiecc} reduce to
\begin{align*}&
\tfrac12(1-\theta)(u_j^R-u_{j+1}^L)
=\tfrac12(1+\theta)(u_j^L-u_{j-1}^R),
\\&
\tfrac12(1+\theta)(f_j^R-f_{j+1}^L)
=\tfrac12(1-\theta)(f_j^L-f_{j-1}^R).
\end{align*}
In matrix notation the \(\gamma=1\) edge condition on field~\(u\) is equivalent to \(V\tilde\uv=\vec{0}\) for vector \(\tilde\uv:=(u_1^R-u_{2}^L,\ldots,u_N^R-u_{1}^L)\), and for \(N\times N\) circulant matrix
\begin{equation*}
V:=\begin{bmatrix}
\tfrac{1}{2}(1-\theta) & 0 & \cdots & 0 & \tfrac{1}{2}(1+\theta)\\
\tfrac{1}{2}(1+\theta) & \ddots &  &  \ddots & 0\\
0 & \ddots  & \ddots & \ddots & \vdots \\
\vdots &  \ddots & \ddots &\ddots &0\\
 0& \cdots& 0 &  \tfrac{1}{2}(1+\theta) & \tfrac{1}{2}(1-\theta)
\end{bmatrix}.
\end{equation*}
The eigenvalues of this circulant matrix are \(\lambda_k=\tfrac{1}{2}(1-\theta)+\tfrac{1}{2}(1+\theta)e^{-2\pi\i k(N-1)/N}\) for \(k=0,\ldots,N-1\) \cite[Sec.~3.1, Eq.~(3.7)]{Gray2006}, and thus only in the case \(k=N/2\) and \(\theta=0\) does~\(V\) have a zero eigenvalue (is singular).
Therefore, the only solution of \(V\tilde\uv=\vec{0}\) is \(\tilde\uv=\vec{0}\) (except for the case \(\theta=0\) and \(N\)~even, when there are also nontrivial zigzag solutions in the nullspace).
A similar argument holds for the flux condition which is equivalent to \(V^T\tilde\fv=\vec{0}\) for vector \(\tilde\fv:=(f_1^L-f_{N}^R,f_2^L-f_{1}^R,\ldots,f_N^L-f_{N-1}^R)\); that is, the only solution is \(\tilde\fv=\vec{0}\) (except the case \(\theta=0\) and \(N\)~is even).
Therefore, for \(\gamma=1\), edge conditions~\cref{eqsiecc} are equivalent to 
\begin{equation}
u_j^R-u_{j+1}^L=0\,,\quad f_j^R-f_{j+1}^L=0
\label{eq:gamma1}
\end{equation}
(except perhaps for the isolated case \(\theta=0\) and \(N\)~even).
Thus \(\gamma=1\) restores full inter-element coupling by requiring  continuity of the field~\(u\) and the flux~\(f\) between elements.

\begin{remark}\label{remzigzag}
For the special case of \(\theta=0\) and even~\(N\), both the continuous edge conditions~\eqref{eq:gamma1} and a discontinuous zigzag mode, \(u_j^R-u_{j+1}^L\propto(-1)^j\) and  \(f_j^R-f_{j+1}^L\propto(-1)^j\), satisfy conditions~\eqref{eqsiecc} at \(\gamma=1\)\,.
Such a macroscale zigzag mode typically occurs with periodic boundary conditions and an even number of grid points.
The zigzag mode is physically deficient because it interpolates from the macro-grid as \(\cos[\pi(x-X_j)/H]+a\sin[\pi(x-X_j)/H]\) for \emph{any} arbitrary coefficient~\(a\).   
All other modes have a unique interpolation in resolved wavenumbers \(|k|\leq\pi/H\)\,.
Because of this physical deficiency, when necessary we restrict attention to the co-dimension two space without the zigzag modes.
Without the zigzag modes, the edge conditions~\cref{eqsiecc} are \text{equivalent to~\eqref{eq:gamma1}.}
\end{remark}

\end{itemize}

For a physically reasonable description of the macroscale dynamics, we need to define some measure of the overall size of the fields~\(u_j\) in each element~\(\XX_j\), \(j=1,\ldots,N\), in order to give physical meaning to the macroscale parameters~\(\Uv=(U_1,\ldots,U_N)\), that is, we need to define the projection~\cU\ invoked by~\cref{eqgridval}. 
Among many, two possibilities are the mid-element value \(\cU [u_j]:=u_j(t,(L_j+R_j)/2)\) (used in \cref{ex:nls}), or the element average \(\cU [u_j]:=\frac1{H_j}\int_{L_j}^{R_j} u_j(t,x)\,dx\)\,.
Any reasonable choice suffices \cite[\S5.3.3]{Roberts2014a}.
The next example implements the element average.

\begin{example}[high-order consistency in discretising diffusion]
\label{eghocdd}
Consider the simplest case of \pde~\cref{eqgenpde}, the case of homogeneous diffusion \(u_t=u_{xx}\)\,, which arises when \(g=0\) and the flux \(f=-u_x\)\,.
\cref{sechochd} describes computer algebra code that discretises this diffusion \pde\ via elements, all of size~\(H\), with inter-element coupling controlled by edge conditions~\cref{eqsiecc}. 
The computer algebra analyses the \pde\ to learn the emergent slow manifold model as a power series in the coupling parameter~\(\gamma\) (\cref{rem:awml}).
Because of the simplicity of the homogeneous diffusion operator, the algebraically learnt slow manifold is here local polynomials. 
In terms of subgrid space variable \(\xi:=(x-L_j)/H\in[0,1]\)\,, tuning parameter~\(\theta\), and centred mean~\(\mu_j\) and difference~\(\delta_j\) operators on the macroscale grid variables~\(U_j\) (\cref{tblopids}), the learnt subgrid structures are, when retaining terms up to \text{quadratic order in~\(\gamma\),}
\begin{subequations}\label{eqsdiffpde}%
\begin{align}
u_j&=U_j+\gamma(\xi-\tfrac12)\mu_j\delta_j U_j 
+ \tfrac12\gamma^2(\xi^2-\xi+\tfrac16)\delta_j^2U_j
\nonumber\\&\quad{} 
-\tfrac14\gamma^2(1+\theta^2)(\xi-\tfrac12)\mu_j\delta_j^3U_j
+\tfrac18\gamma^2(1-\theta^2)(\xi^2-\xi+\tfrac16)\delta_j^4U_j
\nonumber\\&\quad{}
+\theta(\xi-\tfrac12)\big[-\tfrac12(\gamma-\gamma^2)\delta_j^2U_j
+\tfrac14\gamma^2\delta_j^4U_j\big]
+\Ord{\gamma^3}.
\label{eqdudt}
\end{align}
The corresponding learnt lowest-order evolution equation, the closure~\cref{eqcsd}, is
\begin{align}
\partial_t U_j&=\frac{\gamma^2}{H^2}\left[(1-\theta^2)\mu_j^2 
+\theta^2\right]\delta_j^2U_j 
+\Ord{\gamma^3}.
\label{eqdUdtDiff}
\end{align}
This is a correct leading discretisation of the original system for every~\(\theta\)---correct as it is consistent with the homogeneous diffusion \pde\ to errors~\Ord{H^2} as \(H\to0\)\,.
But~\eqref{eqdUdtDiff} has two unusual features for holistic discretisation: it only appears at~\Ord{\gamma^2}; and for tuning \(\theta\neq\pm1\) it invokes the next-nearest neighbours,~\(U_{j\pm2}\) (through~\(\mu_j^2\delta_j^2\)), as well as the two nearest neighbours,~\(U_{j\pm1}\).
\begin{table}
\caption{\label{tblopids}Useful operator identities \protect\cite[p.65, e.g.]{npl61} where \(E^{-1}\) and \(E\) represent a shift to the left and right, respectively, with subscript~\(x\) indicating a shift of distance~\(H\), subscript~\(j\) indicating a shift of one in this element index, and subscript~\(\xi\) indicating a  shift of one in this dimensionless subgrid variable. 
For example, \(\delta_j^2U_j=(E_j-2+E_j^{-1})U_j=U_{j+1}-2U_j+U_{j-1}\) and \(\mu_j\delta_jU_j=\frac12(E_j-E_j^{-1})U_j=\frac12[U_{j+1}-U_{j-1}]\)\,.}
\begin{equation*}
\begin{array}{rlrl}
\hline
&\delta=E^{1/2}-E^{-1/2^{\vphantom{|}}}
&&\mu=\tfrac12(E^{1/2}+E^{-1/2})
\\&\delta_x=2\sinh(H\partial_x/2)
&&\mu_x=\cosh(H\partial_x/2)
\\&\delta_j=2\sinh(\partial_j/2)
&&\mu_j=\cosh(\partial_j/2)
\\&E^{\pm1}=1\pm\mu\delta+\tfrac12\delta^2
&&\mu^2=1+\tfrac14\delta^2
\\[0.5ex]\hline
\end{array}
\end{equation*}
\end{table}%
\cref{seclin} details how centre manifold theory supports such truncations as an approximation to an in-principle exact slow manifold that exists, and is exponentially quickly attractive for a wide range of initial conditions. 
This theoretical support also applies to nonlinear modifications of such linear diffusion.

The algebra of \cref{sechochd} computes to arbitrarily high-order in coupling parameter~\(\gamma\), learning more and more subgrid structures of the diffusion \pde\ \text{\(u_t=u_{xx}\)}\,.
Computed to element interaction errors~\Ord{\gamma^6} we find the macroscale closure~\cref{eqcsd} is 
\begin{align}
H^2\partial_t U_j&=
\theta^2\big[\gamma^2\delta_j^2-\tfrac12\gamma^3\delta_j^4
+\tfrac{1}{4}\gamma^4(\tfrac{5}{3}\delta_j^4+\delta_j^6)-\tfrac{1}{4}\gamma^5(\tfrac{5}{3}\delta_j^6+\tfrac12\delta_j^8)\big]U_j 
\nonumber\\&\quad{}
+\theta^2(1-\theta^2)\big[\tfrac{1}{16}\gamma^4(\tfrac23+\tfrac13\mu_j^2)\delta_j^6
-\tfrac{1}{16}\gamma^5(\tfrac23+\tfrac13\mu_j^2)\delta_j^8 \big]U_j
\nonumber\\&\quad{}
+(1-\theta^2)\big[\gamma^2\mu_j^2\delta_j^2 
-\tfrac12\gamma^3\mu_j^2\delta_j^4
+\tfrac16\gamma^4(1+\tfrac{11}{8}\delta_j^2)\mu_j^2\delta_j^4
\nonumber\\&\qquad{}
-\tfrac16\gamma^5(1+\tfrac{5}{8}\delta_j^{2})\mu_j^2\delta_j^6 \big]U_j
+\Ord{\gamma^6}.  
\label{eqgopdiff}
\end{align}
We test how well such discretisations predict the diffusion dynamics by comparing the \pde\ \(u_t=u_{xx}\) with the equivalent \pde\ of discretisation~\cref{eqgopdiff}.  
To determine the equivalent \pde{}, the macroscale grid fields~\(U_j\) over all discrete~\(j\) are interpolated to provide a description of~\(U_j\) over the continuous variable~\(x\), so that one step in index~\(j\) of field~\(U_j\) is the same as one step in~\(x\) of size~\(H\): \(\delta_j^2U_j=\delta_x^2U_j\) and \(\mu_j\delta_jU_j=\mu_x\delta_xU_j\)\,, and similarly for higher orders.\footnote{Mean operators \(\mu_j\)~and~\(\mu_x\) and difference operators \(\delta_j\)~and~\(\delta_x\) are not equivalent; but they give the same results when operating on macroscale fields~\(U_j\) because the \(x\)~dependence of the macroscale field is defined by an interpolation of the discrete~\(j\) field values (\cref{remcom}).
In contrast, as we show in \cref{secpcmwave1,Xsecpchdd} (in particular~\cref{lemcc1eoca,Xlemcc1eocb}), operators \(\delta_x\)~and~\(\delta_j\) (and similarly \(\mu_x\)~and~\(\mu_j\)) are distinctly different when operating on microscale fields~\(u_j(t,x)\).
}
Upon computing to higher-order errors of~\Ord{\gamma^9}, the equivalent \pde\ to the discrete model is, by expanding the operator identity that \(\delta_x=2\sinh(H\partial_x/2)\) (\cref{tblopids}),
{\renewcommand{\Dn}[3]{\partial_{#1}^{#2}#3}
\begin{align}
\partial_tU&=\gamma^2 \DD xU 
+(1-\theta^2)\left[ (\tfrac13\gamma^2-\tfrac12\gamma^3+\tfrac16\gamma^4)H^2\Dn x4U
\right.\nonumber\\&\quad\left.{}
+(\tfrac{2}{45}\gamma^2-\tfrac{5}{24}\gamma^3+\tfrac{43}{144}\gamma^4-\tfrac{1}{6}\gamma^5+\tfrac{23}{720}\gamma^6)H^4\Dn x6U
\right.\nonumber\\&\quad\left.{}
+(\tfrac{1}{315}-\tfrac{3}{80}\gamma^3+\tfrac{61}{480}\gamma^4-\tfrac{3}{16}\gamma^5+\tfrac{49}{360}\gamma^6-\tfrac{23}{480}\gamma^7+\tfrac{11}{1680})H^6\Dn x8U
\right]
\nonumber\\&\quad{}
+\theta^2\left[ (\tfrac{1}{12}\gamma^2-\tfrac12\gamma^3+\tfrac{5}{12}\gamma^4)H^2\Dn x4U
\right.\nonumber\\&\quad\left.{}
+(\tfrac{1}{360}\gamma^2-\tfrac{1}{12}\gamma^3+\tfrac{23}{72}\gamma^4-\tfrac{5}{12}\gamma^5+\tfrac{8}{45}\gamma^6)H^4\Dn x6U
\right.\nonumber\\&\quad\left.{}
+(\tfrac{1}{20160}\gamma^2-\tfrac{1}{160}\gamma^3+\tfrac{13}{192}\gamma^4-\tfrac{11}{48}\gamma^5+\tfrac{257}{720}\gamma^6-\tfrac{4}{15}\gamma^7+\tfrac{13}{168}\gamma^8)H^6\Dn x8U \right]
\nonumber\\&\quad{}
+\theta^2(1-\theta^2)\left[(\tfrac{1}{16}\gamma^4-\tfrac{1}{16}\gamma^6)H^4\Dn x6U
\right.\nonumber\\&\quad\left.{}
+(\tfrac{1}{48}\gamma^4-\tfrac{1}{16}\gamma^5-\tfrac{1}{192}\gamma^6+\tfrac{3}{32}\gamma^7-\tfrac{3}{64}\gamma^8)H^6\Dn x8U\right]
\nonumber\\&\quad{}
+\theta^4(1-\theta^2)\left[(\tfrac{1}{64}\gamma^6-\tfrac{1}{64}\gamma^8)H^6\Dn x8U\right]+\Ord{H^8\partial_x^9U}.
\label{eqepdediff}
\end{align}
}
For full coupling \(\gamma=1\)  and for every tuning~\(\theta\), all these coefficients evaluate to zero, except for the leading diffusion coefficient which evaluates to one.
Thus the discretisation~\eqref{eqgopdiff} is consistent with \(u_t=u_{xx}\) to error~\Ord{H^8\partial_x^9U}.

Inspection of~\cref{eqepdediff} suggests, and \cref{Xsecpchdd} proves, that a construction of the slow manifold to errors~\ord{\gamma^p}, for every even~\(p\), learns a discrete model consistent with the diffusion \pde\ to errors~\Ord{H^{p}\partial_x^{p+2}U}.
That is, the learnt holistic discretisation is consistent with the diffusion \pde\ \text{to arbitrarily high-order.}
\end{subequations}
\end{example}

In application, this systematic consistency means that one controls and estimates errors in the holistic discretisation by varying the order~\(p\) of the inter-element coupling.

\subsection{Support spatial discretisation with centre manifold theory}
\label{seclin}

We use centre manifold theory \cite[e.g.,][]{Carr81, Haragus2011, Hochs2019} to establish (\cref{thmgsmb}) the existence and emergence of an exact closure~\eqref{eqcsd} to the spatial discretisation of \pde{}s in the class~\cref{eqgenpde}.
Centre manifold theory is based upon either an equilibrium or subspace\slash manifold of equilibria, and primarily follows from the local persistence of a spectral gap in the linearised dynamics \cite[e.g.,][Part~V]{Roberts2014a}.

To find useful equilibria we embed the \pde~\cref{eqgenpde} into a wider class of problems characterised by parameter~\(\gamma\). 
For definiteness of theory, set the macroscale boundary conditions on~\XX\ to be that the field~\(u(t,x)\) is \fL-periodic in space~\(x\).
Recall that \(u_j(t,x)\)~denotes solutions of the \pde~\cref{eqgenpde} on the elements~\(\XX_j\) that partition the domain, so we reserve~\(u(t,x)\), over~\(\tilde\XX\), to denote the union over all elements of such solutions.
Then the edge-element conditions~\cref{eqsiecc} couple each element together.
The information flow between elements is moderated by the coupling parameter~\(\gamma\).
The parameter~\(\gamma\) connects the original \pde\ over the whole domain, \(\gamma=1\), to a useful base problem, \(\gamma=0\)\,.

\begin{lemma}[equilibria] \label{thmequil}
The \pde~\cref{eqgenpde} with \cref{assFandG} on elements~\(\{\XX_j\}\) with edge conditions~\cref{eqsiecc} possesses an \(N\)-dimensional subspace~\EE\ of equilibria parametrised by \(\Uv=(U_1,U_2,\ldots,U_N)\) with parameter values \(\alpha=\gamma=0\)\,. 
Each equilibrium is of a piecewise constant field~\(u^*(x)\) such that  for every \(j\in\JJ\), \(u^*(x)=U_j\) for \(x\in\XX_j\)\,.
\end{lemma}

There usually are other equilibria of~\cref{eqgenpde}+\cref{eqsiecc}: this lemma only addresses the subspace~\EE.

\begin{proof} 
With advection-reaction coefficient \(\alpha=0\) the \pde~\cref{eqgenpde} takes the form \(u_t=-f(x,u,u_x)_x\)\,.
For every piecewise constant field \(u^*(x)=U_j\) in~\(\XX_j\) the gradient \(u^*_{x}=0\)\,, and by \cref{assFandG} the flux \(f(x,u^*,u_{x}^*)=f(x,U_j,0)=0\)\,.
Thus \(\Uv=(U_1,U_2,\ldots,U_N)\) with \(\alpha=0\) are equilibria of the \pde\ on~\(\tilde\XX\). 
The equilibria also require coupling \(\gamma=0\)\,, for which the right-hand sides of the conditions~\cref{eqsiecc} are zero.
The left-hand sides are also zero for fields constant in each element, and so the conditions~\cref{eqsiecc} are satisfied.
\end{proof}

We seek solutions \(u=u^*(x)+\hat u(t,x)\) of the general \pde~\cref{eqgenpde} where \(\hat u(t,x)\)~denotes a small linearised perturbation of the equilubria in~\EE. 
Use \(\hat u_j(t,x)\) as a synonym for~\(\hat u(t,x)\) on the \(j\)th~element~\(\XX_j\).
Then, since parameter \(\alpha=0\) and for small~\(\hat u\) flux \(f(x,u,u_x)\approx f(x,u^*,\hat{u}_x)\)\,, the \pde~\cref{eqgenpde} linearises to
\begin{align}&
\hat u_t=-\partial_x\left[\D{u_x}f(x,u^*,0)\hat u_{x}\right]
\quad\text{on }\XX; 
\nonumber\\&
\text{that is,}\quad 
\hat u_{jt}=\partial_x\big[\kappa_j(x)\hat u_{jx}\big]
\quad\text{for }x\in\XX_j\,,
\label{eqpdelin}
\end{align}
where we define the local diffusivity \(\kappa_j(x):=-\D{u_x}{f(x,U_j,0)}\) for \(x\in\XX_j\)\,.
From~\cref{assFandG}, \(\kappa_j(x)\geq\nu>0\)\,.
The edge conditions~\cref{eqsiecc} are linear, so they apply here with flux~\(f_j=-\kappa_j\hat u_{jx}\)\,, and \(u_j\) replaced by~\(\hat u_j\)---for symbolic simplicity let's omit the hat hereafter.

A spectrum of eigenvalues arises from the general linearised dynamics~\eqref{eqpdelin}.
We turn to determining this spectrum when the coupling parameter \(\gamma=0\)\,, namely the diffusion~\cref{eqpdelin} with isolating edge conditions~\cref{eqsiecc0}.
In doing this, the following self-adjointness is crucial.


\begin{lemma}[self-adjoint]\label{thmsa0}
The diffusion operator \begin{equation}
\cL:=[\kappa_j(x)\partial_x\cdot]_x
\label{eqcL}
\end{equation}
subject to the \(\gamma=0\) edge conditions~\cref{eqsiecc0}, is self-adjoint in the Hilbert space~\LL\ with the usual inner product 
\(\left<v,u\right>:=\sum_j\int_{\XX_j} v_j u_j\,dx\)\,.
\end{lemma}

\begin{proof}
Proceed straightforwardly via integration by parts, where for every~\(u\in\LL\) the flux \(f_j:=-\kappa_j(x)u_{jx}\)\,, and likewise for~\tu\ and~\tf: 
\begin{align*}
\left<\tu,\cL u\right> - \left<\cL\tu, u\right>
&=\sum_j\int_{\XX_j}\tu_j[\kappa_j(x)u_{jx}]_x-u_j[\kappa_j(x)\tu_{jx}]_x\,dx
\\&=\sum_j \big[\tu_j \kappa_j(x)u_{jx}-u_j \kappa_j(x)\tu_{jx}\big]_{L_j}^{R_j}
\\&=\sum_j \big( -\tu_j^R f_j^R+u_j^R \tf_j^R 
+\tu_j^L f_j^L-u_j^L \tf_j^L\big)
\\&=\sum_j \big( -\tu_j^L f_j^L+u_j^L \tf_j^L 
+\tu_j^L f_j^L-u_j^L \tf_j^L\big)
\quad(\text{by }\cref{eqsiecc0})
\\&=0\,.
\end{align*}
\end{proof}

\cref{thmsa} extends \cref{thmsa0} by establishing that the linear operator~\cL, with inter-element coupling controlled by edge conditions~\cref{eqsiecc} is self-adjoint for every~\(\gamma\).

By the self-adjointness of \(\cL\)~\eqref{eqcL}, we need only seek real eigenvalues in the spectrum of~\eqref{eqpdelin}. 
Let's first discuss the zero eigenvalues, and second, the non-zero eigenvalues.
For the diffusion operator~\cref{eqcL} any eigenfunction corresponding to an eigenvalue of zero must be linear on each~\(\XX_j\).
The element-periodicity conditions~\cref{eqsiecc0} for \(\gamma=0\) then require that the eigenfunctions must be constant within every element.
That is, the zero-eigenspace of~\cL\ when \(\gamma=0\) is the \(N\)-D~subspace of equilibria~\EE\ (defined in \cref{thmequil}).
By self-adjointness, there are no generalised eigenfunctions.
Hence the operator~\cref{eqcL}+\cref{eqsiecc0} has \(N\)~eigenvalues of zero, and \EE\ is the corresponding \text{\(N\)-D slow subspace.}

\begin{lemma}[exponential dichotomy] \label{thmdicho}
Under \cref{assFandG}, the operator~\cref{eqcL}, \(\cL=[\kappa_j(x)\partial_x\cdot]_x\) subject to~\cref{eqsiecc0}, in~\LL\ has \(N\)~zero eigenvalues and all other eigenvalues~\(\lambda\) are negative and bounded away from zero by \(\lambda\leq-\beta\) where a bound \(\beta:=\min_{j\in\JJ,\,x\in\XX_j}\big[ 4\pi^2\kappa_j(x)/H_j^2 \big]>0\)\,.
\end{lemma}

\begin{proof} 
The precisely \(N\)~zero eigenvalues are established in the paragraphs preceding \cref{thmdicho}.
Also, \cref{thmsa0} establishes all eigenvalues of~\(\cL\) are real. 
Let \(\lambda\)~be a non-zero eigenvalue and \(v(x)\)~be a corresponding eigenfunction.
Consequently, and using the usual inner product 
\(\left<v,u\right>:=\sum_j\int_{\XX_j} v_j u_j\,dx\)\,,
we proceed to derive the inequality 
\begin{align}
(-\lambda)\|v\|^2
&=-\lambda\left<v,v\right>
=\left<v,-\lambda v\right>
=\left<v,-\cL v\right>
=\sum_{j\in\JJ}\int_{\XX_j}-v_j [\kappa_j(x)v_{jx}]_x\,dx
\nonumber\\&=
\sum_{j\in\JJ}\big[-v_j \kappa_jv_{jx}\big]_{L_j}^{R_j}
+\sum_{j\in\JJ}\int_{\XX_j}v_{jx}\kappa_j v_{jx}\,dx
\quad(\text{integrating by parts})
\nonumber\\&=
\sum_{j\in\JJ} 0
+\sum_{j\in\JJ}\int_{\XX_j}\kappa_j v_{jx}^2\,dx
\quad(\text{using \cref{eqsiecc0}})
\nonumber\\&\geq
\frac\beta{4\pi^2}\sum_{j\in\JJ}\int_{\XX_j}H_j^2 v_{jx}^2\,dx
\,,\label{eqlem3a}
\end{align}
for the bound~\(\beta\) defined in terms of the lowest diffusivity of operator~\(\cL\): \(\kappa_j(x)\geq\beta H_j^2/(4\pi^2)\) for all \(x\in\XX_j\)\,, for every~\(j\in\JJ\)\,.
The bound \(\beta>0\) follows from the monotonicity of~\(f\), \cref{assFandG}, since \(\kappa_j(x) =-\D{u_x}f(x,U_j,0) \geq\nu>0\) for all~\(j\in\JJ\)\,,  so \(\beta/(4\pi^2) =\min_{j\in\JJ,\,x\in\XX_j}\big[\kappa_j(x)/H_j^2\big] \geq \nu/{\max_j H_j^2} >0\)\,.
A first consequence of inequality~\eqref{eqlem3a} is that there are no positive eigenvalues~\(\lambda\).

Secondly, we establish that the non-zero eigenvalues are bounded away from zero by~\(-\beta\).
This bound is established by relating inequality~\eqref{eqlem3a} for a general eigenvalue~\(\lambda\) and associated eigenvector~\(v\) of some heterogeneous problem to the known lowest magnitude eigenvalue~\(\lambda_1\) of a spatially homogeneous problem.
Let \(\cL_*:=H_j^2\DD x{}\) for \(x\in\XX_j\) with edge conditions~\cref{eqsiecc0} (that is, the linear operator~\eqref{eqcL} for the special homogeneous case of \(\kappa_j(x)=H_j^2\)).
Then by the reverse argument to that of the previous paragraph,
\(\sum_{j\in\JJ}\int_{\XX_j}H_j^2 v_x^2\,dx
=\cdots=\left< v,-\cL_* v\right>\)\,.
By the Rayleigh--Ritz theorem, the smallest magnitude, non-zero, eigenvalue~\(\lambda_1\) of~\(\cL_*\) satisfies \(-\lambda_1=\min_{w\perp\EE} \left<w,-\cL_*w\right>/\|w\|^2\)\,, and so 
\(\sum_{j\in\JJ}\int_{\XX_j}H_j^2 v_x^2\,dx
=\left< v,-\cL_* v\right> \geq -\lambda_1\|v\|^2\)\,.
But for~\(\cL_*\), subject to the condition of \(H_j\)-periodicity~\cref{eqsiecc0}, it is well-known that the smallest magnitude, non-zero eigenvalue is \(\lambda_1=-4\pi^2\) (corresponding to eigenmodes \(\e^{\pm\i2\pi x/H_j}\)).
Hence, inequality~\eqref{eqlem3a} becomes
\((-\lambda)\|v\|^2
\geq \frac\beta{4\pi^2}\sum_{j\in\JJ}\int_{\XX_j}H_j^2 v_x^2\,dx
\geq -\frac\beta{4\pi^2}\lambda_1\|v\|^2
= \beta\|v\|^2 \)\,.
That is, \(\lambda\leq-\beta\) for every non-zero eigenvalue~\(\lambda\).
\end{proof}

With \cref{thmdicho} proving a spectral gap in the linearised dynamics of~\cref{eqpdelin}+\cref{eqsiecc0} with \(\alpha=\gamma=0\)\,, we are now in a position to discuss the corresponding slow manifold of~\cref{eqgenpde}+\cref{eqsiecc}.  
Inspired by the backwards theory of numerical analysis \cite[e.g.,][]{Grcar2011}, we apply backwards theory to establish the existence and emergence of a slow manifold \cite[]{Hochs2019,Roberts2018a}.

For rigorous theory we notionally adjoin the two trivial dynamical equations \(\alpha_t=\gamma_t=0\) to the linearised system~\cref{eqpdelin}+\cref{eqsiecc}.
Then the equilibria in \(\alpha\gamma u\)-space are~\((0,0,u^*(x))\) since these equilibria are at \(\alpha=\gamma=0\)\,.
Hence, strictly, there are two extra zero eigenvalues associated with the trivial \(\alpha_t=\gamma_t=0\)\,, and the corresponding slow subspace of each equilibria is \((N+2)\)-D.
But, except for issues associated with the domain of validity for non-zero \(\alpha\) and \(\gamma\), for simplicity, in the following we do not explicitly discuss these two trivial dynamical equations nor their eigenvalues, but consider them implicit.

\begin{theorem}[slow manifold]\label{thmgsmb}
Consider the generic reaction-advection-diffusion \pde~\cref{eqgenpde} with \cref{assFandG} on domain~\(\tilde\XX\) with inter-element edge conditions~\cref{eqsiecc}. 
For every order~\(p\geq2\)\,, there exists a system for~\(u(t,x)\), 
namely the following polynomial map~\(\upsilon_p\) and \pde\ with polynomial nonlinearity~\(\ff_p\)
\begin{equation}
u(t,x)=u^*+\upsilon_p(t,x,v(t,x)) 
\quad\text{where }v_t=\cL v+\ff_p(t,x,v),
\label{eqpdesd}
\end{equation}
such that the map+\pde~\cref{eqpdesd} is \Ord{\|u-\EE\|^p+\alpha^p+\gamma^p} close to \cref{eqgenpde}+\cref{eqsiecc}, and that the map+\pde~\cref{eqpdesd} has the following properties.
\begin{enumerate}[label={\ref{thmgsmb}}(\alph{enumi}), leftmargin=*
                 ,ref={\cref{thmgsmb}}(\alph{enumi})] %
\item There exists an open domain~\cE, in \(\alpha\gamma u\)-space, containing the subspace~\EE, such that in~\cE\ there exists a constructible slow manifold of~\pded, polynomial in~\Uv, \(\alpha\) and~\(\gamma\) of order~\(p-1\):
\begin{equation}
u=u(t,x,\Uv,\gamma,\alpha)
\quad\text{such that}\quad
\frac{d\Uv}{dt}=\Gv(t,\Uv,\gamma,\alpha).
\label{eqgensm}
\end{equation}

\item\label{thmgsm:emergent} 
This slow manifold is emergent in the sense that for all solutions~\(u(t,x)\) of~\pded\ that stay in~\cE, there exists a solution~\(\Uv(t)\) of~\cref{eqgensm} such that \(u(t,x)=u(t,x,\Uv(t),\gamma,\alpha)+\Ord{e^{-\beta' t}}\) for some decay rate \(0<\beta'<\beta\) (\(\beta\) defined by \cref{thmdicho}).

\end{enumerate}
\end{theorem}

In this theorem, and subsequently, asserting a quantity~\(\epsilon=\Ord{v^p+\alpha^q+\gamma^r}\) is to mean the asymptotic property that \(\epsilon/(v^p+\alpha^q+\gamma^r)\) is bounded as \((v,\alpha,\gamma)\to\vec 0\) \cite[e.g.,][]{Roberts2014a}.
Also, the norm \(\|u-\EE\|\) is a measure of the distance of~\(u\) from the subspace~\EE\ of equilibria.
Strictly, this measure of distance is a norm in appropriate graded Frechet spaces that are obtained from intersections of compactly nested sequences of Banach spaces \cite[Defs.~2.2--2.4]{Hochs2019} (hereafter denoted~HR).

\begin{proof}
For every equilibria \(u^*\in\EE\)\,, we invoke the backwards Theorems~2.18 and~2.22 of~HR.
From \cref{thmsa0}, the eigenfunctions of the operator \(\cL=[\kappa_j(x)\partial_x\cdot]_x\) subject to~\cref{eqsiecc0} form an orthonormal basis of the requisite Hilbert space on \(\tilde\XX:=\cup_j\XX_j\) (Hilbert--Schmidt theorem).
Since \(f\) and~\(g\) are smooth (infinitely differentiable) they can be expanded polynomially to any specified order~\(p\) as required.
Hence, with \cref{assFandG}, the conditions of Theorems 2.18~and~2.22 by HR are satisfied.
Thus Corollary~2.23 \text{by~HR applies.}

For every order~\(p\), and for~\(u(t,x)\) defined by the map+\pde~\eqref{eqpdesd}, this corollary asserts that \(u(t,x)\) satisfies \cref{eqgenpde}+\cref{eqsiecc} to a residual \(R_p=\Ord{\|(v,\alpha,\gamma)\|^p}=\Ord{\|u-\EE\|^p+\alpha^p+\gamma^p}\)\,.
In this sense of asymptotic agreement, every such system~\cref{eqpdesd} is close to the system \cref{eqgenpde}+\cref{eqsiecc}.

Also, every system~\pded\ is constructed to have clear invariant subspaces in~\(v\), and so clear invariant manifolds in~\(u\) (Definition~2.11 of~HR).
Further, from the exponential dichotomy of \cref{thmdicho} there exists a domain~\(D_{\mu}\), in \(\alpha\gamma u\)-space,  in which the system~\pded\ has a slow manifold.
Since \(0\leq \mu<\beta\leq -\lambda\) (Section~2.3.1 of~HR) and by Proposition~2.10 of~HR, the slow manifold solutions are exponentially quickly attractive, at rate of at least~\(\beta'=\beta-\mu>0\)\,, to all solutions in the domain~\(D_{\mu}\).

These properties hold when based upon each and every equilibrium~\(u^*\) in the slow subspace~\EE. 
By smoothness of the system \cref{eqgenpde}+\cref{eqsiecc}, we can ensure we construct the close systems \pded\ to be smooth as a function of~\(u^*\).
Hence the properties hold smoothly in the union of all the domains~\(D_\mu\), to form a global domain, denoted~\(\cE\), in which the properties hold. 
\text{This establishes \cref{thmgsmb}.}
\end{proof}

\cref{thmgsmb} holds for the \pde\ system \cref{eqgenpde} on domain~\(\tilde\XX\) with edge conditions~\cref{eqsiecc} between elements.
The system \cref{eqgenpde}+\cref{eqsiecc} reverts to the original  \pde~\cref{eqgenpde} on the domain~\(\XX\) when we evaluate \cref{eqgenpde}+\cref{eqsiecc} on~\(\tilde\XX\) at full coupling \(\gamma=1\)\,.
Thus the evolution equation~\cref{eqgensm} evaluated at full coupling, namely \(d\Uv/dt=\Gv(t,\Uv,1,\alpha)\)\,, is controllably close to a discretisation of the original \pde.
Further, this evolution is that on a slow manifold which is emergent in the domain \(\cE_1:=\cE|_{\gamma=1}\) (which is not empty as \(\cE_1\)~contains at least the equilibria \(u={}\)constant and \(\alpha=0\)).

Similar slow manifold properties to those of \cref{thmgsmb} could be established via the \emph{forward} Theorems~2.9 and~3.22 of \cite{Haragus2011}, together with Proposition~3.6 of \cite{Potzsche2006}.  
The argument for Theorem~6 by \cite{Jarrad2016a} is an example.
However, such a forward approach attempts to prove the existence of an invariant manifold for the nonlinear system~\cref{eqgenpde}+\cref{eqsiecc}, and in doing so imposes significant constraints on the functions~\(f\) and~\(g\), constraints that in applications either are often hard to establish, or are not satisfied. 
The \emph{backwards} theory of \cite{Hochs2019}, which shows the existence of a system with an invariant manifold arbitrarily close to \cref{eqgenpde}+\cref{eqsiecc}, is less restrictive on the functions~\(f\) and~\(g\) and so has a wider range of applications.

\subsection{The slow manifold of wave-like PDEs}
\label{secsmwpde}

Although this article's principal scope is the spatial discretisation, or dimensional reduction, of reaction-advection-diffusion \pde{}s~\cref{eqgenpde}, much of the approach and theory usefully applies to the spatial discretisation of non-autonomous wave-like \pde{}s in the form
\begin{equation}
u_{tt}=-f(x,u,u_x)_x+\alpha g(t,x,u,u_x)
\label{eqwavpde}
\end{equation}
on a domain~\XX, with \cref{assFandG}.
The only difference between~\cref{eqwavpde,eqgenpde} is the number of time derivatives on the left-hand side.
This subsection comments on the similarities and differences of the theoretical support for such wave systems.

Partition space into \(N\)~elements as described at the beginning of \cref{secsapc} with the  \(j\)th~element over the interval \(\XX_j=(L_j,R_j)\)\,, and apply the edge conditions~\cref{eqsiecc}.
Then, for \(\alpha=\gamma=0\)\,, the subspace~\EE\ of piecewise linear equilibria of \cref{thmequil} still exists.
Upon linearisation of~\cref{eqwavpde} about each of these equilibria, the spatial differential operator \(\cL=[\kappa_j(x)\partial_x\cdot]_x\) on~\(\tilde\XX=\cup_j\XX_j\) remains self-adjoint.
The exponential dichotomy of the operator~\cL\ (\cref{thmdicho}) still applies, namely that there are \(N\)~eigenvalues of zero, and the others are\({}\leq-\beta\)\,.
So far, the considerations are the same for both diffusion systems and wave systems.

Differences between wave-like \pde{}s~\cref{eqwavpde} and reaction-advection-diffusion \pde{}s~\cref{eqgenpde} start with the eigenvalues of the right-hand side operator~\cL\ with its eigenvalues denoted by~\(\lambda\).
Seeking wave solutions of~\eqref{eqwavpde} with frequency~\(\omega\), then \(\omega=\pm\sqrt{-\lambda}\) and all frequencies~\(\omega\) are real as every \(\lambda\leq0\)\,.
Here the slow manifold dichotomy is now between slow waves with near zero frequency, separated from fast waves with frequencies\({}\geq\sqrt{\beta}\) (\cref{thmsa0}).
Such subcentre slow manifolds (a class of nonlinear normal modes \cite[e.g.]{Shaw94a}) are ubiquitous in both geophysical and elastic engineering applications.
However, much less is known rigorously about subcentre slow manifolds: even their existence is problematic \cite[]{Lorenz87}.
Nonetheless, inspired by the backwards theory of \cite{Roberts2018a, Hochs2019}
we make the following backwards conjecture for the wave \pde~\cref{eqwavpde} that is analogous to \cref{thmgsmb} for dissipative systems.

\begin{conjecture}\label{thmwavemap}
For every order~\(p\), there exists a (polynomial) coordinate transformation and a (polynomial) \pde\ system in the new variables~\((\Uv,V(x))\) of the form
\begin{subequations}\label{eqwavct}%
\begin{align}
&u=u(t,x,\Uv,V,\gamma,\alpha),\quad
\label{eqwavcta}
\\&
\dd t\Uv=\Gv(t,\Uv,V,\gamma,\alpha),\quad \DD tV=H(t,x,\Uv,V,\gamma,\alpha)V,
\label{eqwavctb}
\end{align}
\end{subequations}
such that  in the \(\alpha\gamma u\)-space the system~\eqref{eqwavct} is the same as the \pde~\cref{eqwavpde} to a residual~\Ord{\|u-\EE\|^p+\gamma^p+\alpha^p}, and such that the subspace~\EE\ is the tangent space of the manifold \(V=0\) at \(\gamma=\alpha=0\).
\begin{enumerate}[label={\ref{thmwavemap}}(\alph{enumi}), leftmargin=*
                 ,ref={\cref{thmwavemap}}(\alph{enumi})] %
\item Let \cE\ denote a \(\alpha\gamma u\)-domain in which the coordinate transform~\cref{eqwavcta} is a diffeomorphism containing~\EE, then \(V=0\) is an exact slow manifold of the dynamics of~\cref{eqwavct}: that is, a slow manifold is
\begin{equation}
u=u(t,x,\Uv,0,\gamma,\alpha)\quad \text{such that}\quad
\dd t\Uv=\Gv(t,\Uv,0,\gamma,\alpha).
\label{eqwavsm}
\end{equation}

\item Solutions near, but off the slow manifold with \(V\neq 0\)\,, generally evolve differently (due to wave-wave forcing of mean flow):
\begin{equation*}
\dd t\Uv=\Gv(t,\Uv,0,\gamma,\alpha)+\Ord{\|V\|^2}.
\end{equation*}

\end{enumerate}
\end{conjecture}

Consequently, we contend that the methodology developed here for learning spatially discrete, finite dimensional, accurate, models of dissipative \pde{}s may be also usefully applied to many wave-like \pde{}s~\cref{eqwavpde}, provided we accept the potential for microscale wave-wave interactions be unresolved in the model.
In order to resolve such wave-wave forcing of the mean flow, one would have to include the amplitude modulation of the fast waves into the list of macroscale variables of interest.

\subsection{Preserving self-adjointness}
\label{secpsa}

We now return to reaction-advection-diffusion \pde{}s~\cref{eqgenpde} embedded into spatial elements by the coupling conditions~\cref{eqsiecc} (for general \(\gamma\)).
\cref{thmgsmb} proves the existence of systems arbitrarily close to~\cref{eqgenpde}+\cref{eqsiecc}---systems that have a constructible slow manifold that forms an emergent discretisation of~\cref{eqgenpde}. 
However, \cref{thmgsmb} does not address how well the constructed slow manifold discretisation models the original \pde{}~\cref{eqgenpde}.
This article proves two crucial properties of the slow manifold modelling. 
Firstly, this section establishes that the particular edge conditions~\cref{eqsiecc} preserve self-adjointness and thus preserves this fundamental symmetry in the discretisation. 
Secondly, \cref{secphoc1} proves that the slow manifold of \cref{eqgenpde}+\cref{eqsiecc}, the discrete model, is \emph{consistent} \text{to the \pde~\cref{eqgenpde}.}

Self-adjointness is a property of linear operators, but the \pde~\cref{eqgenpde} is nonlinear.
Consequently, we establish self-adjointness for linear perturbations about a set of base states of~\cref{eqgenpde}, with coefficient~\(\alpha=0\) to eliminate nonlinearity~\(g\) from consideration.
That is, as before, we linearise \pde{}~\cref{eqgenpde} about the equilibria \(u_j(t,x)=U_j\) constant in each element~\(\XX_j\) (\cref{thmequil}), to obtain the linearised \pde{}~\eqref{eqpdelin}.
\cref{thmsa0} established self-adjointness of the diffusion operator~\(\cL\) of~\eqref{eqpdelin} subject to the \(\gamma=0\) edge conditions~\cref{eqsiecc0}, and the following \cref{thmsa} generalises it to edge conditions~\cref{eqsiecc} with \(\gamma\neq0\). 

\begin{lemma}[self-adjoint] \label{thmsa}
The differential operator of the system~\cref{eqpdelin,eqsiecc}, namely \(\cL=[\kappa_j(x)\partial_x\cdot]_x\) on~\LL\ and subject to~\cref{eqsiecc}, and in the usual inner product \(\left<\tu,u\right>:=\sum_{j\in\JJ}\int_{\XX_j} \tu_j u_j\,dx\)\,, is self-adjoint for almost every real~\(\gamma\) and~\(\theta\). 
\end{lemma}

The exception in ``almost every'' is the isolated case when the number of elements is even, \(\gamma=1\)\,, and \(\theta=0\) that is associated with the unphysical nature of the zigzag mode as discussed by \cref{remzigzag}.


\begin{proof}
Let's consider the edge conditions of the \(N\)~elements as \(N\times N\)~matrix equations with coupling matrices~\(C_\pm\). 
For each of the patch right and left edges, \(e\in\{R,L\}\) respectively, define the field vector \(\uv^{e}:=(u_1^{e},\ldots,u_N^{e})\) and the flux vector \(\fv^{e}:=(f_1^{e},\ldots,f_N^{e})\)\,. 
Then, using superscript~\(T\) to denote matrix transpose, the edge conditions~\eqref{eqsiecc} are
\begin{align}
&C_+\uv^R=C_-^T\uv^L,\quad 
C_-\fv^R=C_+^T\fv^L,\label{eq:LRcc}
\quad\text{for circulant matrices}
\\&\nonumber
C_\pm:=\begin{bmatrix}
1-\tfrac{\gamma}{2}(1\pm\theta) & 0 & \cdots & 0 & \tfrac{\gamma}{2}(1\pm\theta)\\
 \tfrac{\gamma}{2}(1\pm\theta) & \ddots & \ddots  &  & 0\\
0 & \ddots  & \ddots & \ddots & \vdots \\
\vdots &  \ddots & \ddots &\ddots &0\\
 0& \cdots& 0 &  \tfrac{\gamma}{2}(1\pm\theta) & 1-\tfrac{\gamma}{2}(1\pm\theta) 
\end{bmatrix}.
\end{align}
We firstly prove that these edge conditions preserve self-adjointness when either~\(C_+\) or~\(C_-\) is invertible, and then secondly when neither is invertible.
This proof requires the commutativity  
\begin{equation}
C_+^TC_-=C_-C_+^T. \label{eq:id}
\end{equation}
This commutativity follows since~\(C_\pm\), and their transposes, are circulant \cite[Thm.~3.1(1)]{Gray2006}, and holds for any two circulant matrices.

First consider the case when \(C_+\) is invertible. 
Then \(C_+^T\) is also invertible and circulant, so by commutativity~\eqref{eq:id}, \(C_-C_+^{-1,T}=C_+^{-1,T}C_-\)\,.
Now, for every~\(u,\tu\in\LL\)\,,
\begin{align*}
&\left<\tu,\cL u\right> - \left<\cL\tu, u\right>
\\&=\sum_{j\in\JJ}\int_{\XX_j}\tu_j[\kappa_j(x)u_{jx}]_x-u_j[\kappa_j(x)\tu_{jx}]_x\,dx
\\&\qquad(\text{then integrate by parts})
\\&=\sum_{j\in\JJ} \big[\tu_j \kappa_j(x)u_{jx}-u_j \kappa_j(x)\tu_{jx}\big]_{L_j}^{R_j}
\\&\qquad(\text{then as flux }f_j:=-\kappa_j(x)u_{jx}\text{ and }\tf_j:=-\kappa_j(x)\tu_{jx})
\\&=\sum_{j\in\JJ} \big[ -\tu_j^R f_j^R+u_j^R \tf_j^R 
+\tu_j^L f_j^L-u_j^L \tf_j^L\big]
\\&=-\tilde{\uv}^{R,T}\fv^R+\uv^{R,T}\tilde{\fv}^R+\tilde{\uv}^{L,T}\fv^L-\uv^{L,T}\tilde{\fv}^L
\\&\qquad(\text{then, since \(C_+\) is invertible})
\\&=-\tilde{\uv}^{R,T}C_+^{T}C_+^{-1,T}\fv^R+\tilde{\uv}^{L,T}\fv^L+\uv^{R,T}C_+^{T}C_+^{-1,T}\tilde{\fv}^R-\uv^{L,T}\tilde{\fv}^L
\\&=-[C_+\tilde{\uv}^{R}]^TC_+^{-1,T}\fv^R+\tilde{\uv}^{L,T}\fv^L+[C_+\uv^{R}]^TC_+^{-1,T}\tilde{\fv}^R-\uv^{L,T}\tilde{\fv}^L
\\&\qquad(\text{apply edge conditions~\eqref{eq:LRcc}})
\\&=-[C_-^T\tilde{\uv}^{L}]^TC_+^{-1,T}\fv^R+\tilde{\uv}^{L,T}\fv^L+[C_-^T\uv^{L}]^TC_+^{-1,T}\tilde{\fv}^R-\uv^{L,T}\tilde{\fv}^L
\\&=-\tilde{\uv}^{L,T}C_-C_+^{-1,T}\fv^R+\tilde{\uv}^{L,T}\fv^L+\uv^{L,T}C_-C_+^{-1,T}\tilde{\fv}^R-\uv^{L,T}\tilde{\fv}^L
\\&\qquad(\text{apply commutativity~\eqref{eq:id}})
\\&=-\tilde{\uv}^{L,T}C_+^{-1,T}C_-\fv^R+\tilde{\uv}^{L,T}\fv^L+\uv^{L,T}C_+^{-1,T}C_-\tilde{\fv}^R-\uv^{L,T}\tilde{\fv}^L
\\&\qquad(\text{apply edge conditions~\eqref{eq:LRcc}})
\\&=-\tilde{\uv}^{L,T}C_+^{-1,T}C_+^T\fv^L+\tilde{\uv}^{L,T}\fv^L+\uv^{L,T}C_+^{-1,T}C_+^T\tilde{\fv}^L-\uv^{L,T}\tilde{\fv}^L
\\&=0\,.
\end{align*}
Thus \(\cL\) is self-adjoint when \(C_+\)~is invertible.

Second, consider the case when \(C_-\)~is invertible.
The detailed argument directly corresponds to the previous case, so we present a summary:
\begin{align*}
&\langle \tilde{u},\mathcal{L}u\rangle-\langle \mathcal{L}\tilde{u},u\rangle
\\&=-\tilde{\uv}^{R,T}\fv^R+\uv^{R,T}\tilde{\fv}^R+\tilde{\uv}^{L,T}\fv^L-\uv^{L,T}\tilde{\fv}^L
\\&=-\tilde{\uv}^{R,T}\fv^R+[C_-^T\tilde{\uv}^{L}]^TC_-^{-1}\fv^L+\uv^{R,T}\tilde{\fv}^R-[C_-^T\uv^{L,T}]C_-^{-1}\tilde{\fv}^L\
\\&=-\tilde{\uv}^{R,T}\fv^R+[C_+\tilde{\uv}^{R}]^TC_-^{-1}\fv^L+\uv^{R,T}\tilde{\fv}^R-[C_+\uv^{R}]^TC_-^{-1}\tilde{\fv}^L
\\&=-\tilde{\uv}^{R,T}\fv^R+\tilde{\uv}^{R,T}C_-^{-1}C_+^T\fv^L+\uv^{R,T}\tilde{\fv}^R-\uv^{R,T}C_-^{-1}C_+^T\tilde{\fv}^L
\\&=-\tilde{\uv}^{R,T}\fv^R+\tilde{\uv}^{R,T}C_-^{-1}C_-\fv^R+\uv^{R,T}\tilde{\fv}^R-\uv^{R,T}C_-^{-1}C_-\tilde{\fv}^R
\\&=0\,.
\end{align*}
Thus \(\cL\) is self-adjoint when \(C_-\)~is invertible.

Now consider the case that neither \(C_+\)~nor~\(C_-\) is invertible, so that their determinants are zero.  
By a first row expansion,  \(\det C_\pm=\left[1-\tfrac{\gamma}{2}(1\pm\theta)\right]^N-(-1)^{N}\left[\tfrac{\gamma}{2}(1\pm\theta)\right]^N\)\,.
The determinant is zero only when \(\left[1-\tfrac{\gamma}{2}(1\pm\theta)\right]^N=\left[-\tfrac{\gamma}{2}(1\pm\theta)\right]^N\)\,.
This cannot occur in the case \(\gamma(1\pm\theta)=0\) as then the \(\rhs=0\) and the \(\lhs=1\): consequently we can divide by the \rhs\ and seek when 
\(\big[1-\frac2{\gamma(1\pm\theta)}\big]^N=1\)\,.
For real~\(\gamma,\theta\), the zero determinant thus can only occurs when \(N\)~is even and \(1-\frac2{\gamma(1\pm\theta)}=-1\)\,.
That is, when \(N\)~is even and \(\gamma(1\pm\theta)=1\)\,.
Hence the previous arguments which depend on either \(C_+\)~or~\(C_-\) being invertible only fail when \(N\)~is even and both \(\gamma(1\pm\theta)=1\)\,.
This last pair of equations is only satisfied when \(\gamma=1\) and \(\theta=0\)\,.
Therefore, the operator~\cL\ on~\LL\ is self-adjoint for every real~\(\gamma\) and every real~\(\theta\), except perhaps for the isolated specific case of even~\(N\), \(\gamma=1\) and \(\theta=0\)\,.
%
\end{proof}

The proofs of this section hold for periodic boundary conditions for~\(u(t,x)\) on~\XX.
We expect further research to establish self-adjointness for other common boundary conditions for~\(u(t,x)\) on~\XX, such as Dirichlet and Neumann.

\section{Prove high-order consistency in 1D space}
\label{secphoc1}

\cref{seclin} establishes the in-principle existence and emergence of an exact discrete closure to the dynamics of \pde{}s in the class~\cref{eqgenpde}, a closure that also preserves self-adjointness (\cref{secpsa}).
An issue is whether the systematic approximations, established by \cref{thmgsmb}, are accurate.
As evidence for their accuracy, this section proves that the  constructed approximate spatially discrete models are consistent to the original \pde~\cref{eqgenpde} to any specified \text{order of error.}

As a `stepping stone' to the second-order diffusion \pde, and because of its interest in its own right, \cref{secpcmwave1} proves consistency of cognate discrete models of a first-order wave \pde.
\cref{Xsecpchdd} then invokes similar arguments to prove consistency of discrete models of diffusion-like \pde{}s.
In this section we restrict attention to homogeneous and autonomous \pde{}s, that is, \(f\) and~\(g\) are here assumed independent of~\(t,x\).

\cref{sechhd} uses an embedding to extend the consistency proof to the homogenisation of systems with microscale heterogeneity.

\subsection{Example: consistent discretion of the first-order wave~PDE}
\label{sechdudwpde}

As a first step, corresponding to \cref{eghocdd} for diffusion, here we analogously construct discrete models of the first-order wave \pde\ 
\begin{equation}
u_t+cu_x=0\,.
\label{equniwpde}
\end{equation}
The spatially discrete models are obtained using the elements and coupling introduced in \cref{secsapc}, and are constructed by computer algebra code (\cref{sechochwave1}) that applies deep recursive refinement to algebraically learn the spatial discretisations as a slow manifold in a power series in inter-element coupling~\(\gamma\).
The code also verifies the high-order consistency between the discrete model and the wave \pde~\eqref{equniwpde}.
In addition, this example explores the upwind\slash downwind character induced by the parameter~\(\theta\) in the inter-element \text{coupling edge conditions~\cref{eqsiecc}.}

A major difference between the diffusion operator \(\cL=\partial_x[\kappa_j(x)\partial_x]\) in~\cref{eqcL} and the advection operator \(-c\partial_x\) in~\cref{equniwpde} is that (albeit depending upon boundary conditions) the diffusion operator is self-adjoint and the advection operator is not (it is anti-symmetric).
Thus the issue of self-adjointness is \text{not discussed here.}

For the wave \pde~\eqref{equniwpde} with inter-element edge conditions~\cref{eqiecca} (flux coupling~\eqref{eqieccb} does not apply here), the subspace~\EE\ of piecewise constant equilibria still exists.  
A proof is a simpler version of that given for \cref{thmequil}: for every piecewise constant field \(u^*(x)=U_j\) in~\(\XX_j\) the wave \pde~\eqref{equniwpde} is simply \(u_t=0\) and thus \(\Uv=(U_1,U_2,\ldots,U_N)\) are equilibria of the \pde\ on~\(\tilde\XX\) with conditions~\cref{eqiecca} satisfied for \(\gamma=0\)\,.

For reasons corresponding to those discussed by \cref{secsmwpde}, the spatial discretisation of the wave \pde~\eqref{equniwpde} arises as a slow subcentre manifold on coupled elements.
Since~\(u_j(t,x)\) denotes the field~\(u(t,x)\) on the \(j\)th~element,  \(u_j\)~is to satisfy the wave \pde~\eqref{equniwpde}, namely \(u_{jt}=-cu_{jx}\) on~\(\XX_j\).
Let the elements be coupled with edge conditions~\cref{eqiecca} with arbitrary coupling parameter~\(\gamma\).
Recall that the measure of~\(u_j(t,x)\) in each element is almost immaterial \cite[Lemma~5.1]{Roberts2014a}, so here we choose each macroscale parameter~\(U_j(t)\) to be the element average.

For the case of equi-sized elements, \(H_j=H\)\,, the code of \cref{sechochwave1} constructs the slow manifold as a power series in inter-element coupling parameter~\(\gamma\). 
The tuning parameter~\(\theta\) provides a range of alternative discrete models for the macroscale dynamics.
The code finds the wave \pde~\eqref{equniwpde} with inter-element coupling~\cref{eqiecca} has a slow manifold with evolution governed \text{by the \ode{}s}
\begin{subequations}\label{eqepdewave}
\begin{align}
\dot U_j&=-\frac{\gamma c}{2H}(U_{j+1}-U_{j-1})+\theta\frac{\gamma c}{2H}(U_{j+1}-2U_j+U_{j-1})+\Ord{\gamma^2},
\label{eqlowgamma}
\end{align}
in terms of the element average values~\(U_j\).
At full coupling \(\gamma=1\)\,, the parameter~\(\theta\) includes the following alternatives:
\(\theta=0\) gives the anti-symmetric, centred difference, form \(\dot U_j\approx-c(U_{j+1}-U_{j-1})/(2H)\);
\(\theta=1\) gives the upwind (when \(c>0\)) difference \(\dot U_j\approx- c(U_j-U_{j-1})/H\)\,;
whereas \(\theta=-1\) gives the downwind (when \(c>0\)) difference \(\dot U_j\approx- c(U_{j+1}-U_j)/H\) (and vice versa when \(c<0\)).

Such discrete models of the macroscale evolution are learnt in conjunction with the subgrid field structure of the slow manifold.
The learnt subgrid field corresponding to the low-order~\eqref{eqlowgamma} is the linear \(u_j=U_j
+\gamma(\xi-\tfrac12)(\mu_j\delta_j-\tfrac12\theta_j\delta_j^2)U_j
+\Ord{\gamma^2}\)\,,
written in terms of the local subgrid variable \(\xi:=(x-L_j)/H\), and centred mean~\(\mu_j\) and difference~\(\delta_j\) operators (\cref{tblopids}).
Higher-order models have more detailed subgrid fields representing more effects of both inter-element communication and subgrid scale physics. 

The computer algebra of \cref{sechochwave1} easily computes to higher-order in~\(\gamma\).
Deep iterative refinement learns that the macroscale evolution~\eqref{eqlowgamma} on the slow manifold is refined to the following, where~\(U_j\) is omitted for brevity:
\begin{align}
\tfrac{H}{c}\partial_t&
= -\gamma\mu_j\delta_j
+\tfrac12\theta\gamma(1-\gamma)\delta_j^2
+\tfrac14\gamma^2\big[1-\tfrac13\gamma+\theta^2(1-\gamma)\big]\mu_j\delta_j^3
\nonumber\\&\quad{}
-\tfrac18\theta\gamma^2(1-\gamma)(2-\gamma-\theta^2\gamma)\delta_j^4
-\tfrac1{16}\gamma^3\big[\tfrac43-\gamma +\theta^2(4-6\gamma)-\theta^4\gamma\big]\mu_j\delta_j^5
\nonumber\\&\quad{}
+\Ord{\gamma^5,\delta_j^6}
.\label{eqgopwave1}
\end{align}
The \(\gamma\)-dependence has the appealing feature that when evaluated at full coupling a derived discrete model such as~\eqref{eqgopwave1} is independent of the tuning~\(\theta\) up to some order in~\(\delta\): 
here, for example, at \(\gamma=1\) the evolution~\eqref{eqgopwave1} becomes simply \(\tfrac Hc\partial_t=-\mu_j\delta_j+\tfrac16\mu_j\delta_j^3+\Ord{\delta_j^5}\)\,.
That such a discrete model is consistent with the wave \pde~\eqref{equniwpde} is seen by the equivalent \pde\ to the discrete model.
For example, constructing~\eqref{eqgopwave1} to next order, errors~\Ord{\gamma^6}, and substituting expansions for \(\delta_j=\delta_x=2\sinh(H\partial_x/2)\) and \(\mu_j=\mu_x=\cosh(H\partial_x/2)\) (\cref{tblopids}; when operating on macroscale grid field~\(U_j\)), the equivalent \pde\ of the \text{discrete model~\eqref{eqgopwave1} is}
{\renewcommand{\Dn}[3]{\partial_{#1}^{#2}#3}
\begin{align}
\frac1c \partial_tU&
=-\gamma \partial_xU
+\tfrac12\gamma(1-\gamma)\theta H\DD xU
-\tfrac1{12}\gamma(1-\gamma)(2-\gamma-3\theta^2\gamma)H^2\Dn x3U
\nonumber\\&\quad{}
+\tfrac1{24}\gamma(1-\gamma)\theta\big[1-6\gamma +3(1+\theta^2)\gamma^2\big]H^3\Dn x4U
\nonumber\\&\quad{}
-\tfrac1{240}\gamma(1-\gamma)\big[2 -(13+15\theta^2)\gamma 
+12(1+5\theta^2)\gamma^2 
\nonumber\\&\qquad{}
-3(1+10\theta^2+15\theta^4)\gamma^3 \big]H^4\Dn x5U
\quad{}+\Ord{H^5\partial_x^6U}
.\label{eqepdewave1}
\end{align}
}
\end{subequations}
At full coupling \(\gamma=1\)\,, due to the factors~\((1-\gamma)\), this equivalent \pde\ reduces to the required wave \pde~\eqref{equniwpde} for every value of~\(\theta\).
The order of error in this consistency appears proportional to the chosen order of \text{coupling in~\(\gamma\).}

\paragraph{Perturbations}
One can perturb the wave \pde~\eqref{equniwpde} and the computer algebra still constructs a discrete model that is consistent to high-order to the perturbed \pde.
The code of \cref{sechochwave1} caters for perturbations to the first-order wave \pde~\eqref{equniwpde} such as
\begin{equation*}
\D tu+c\D xu=\alpha\left[c_0u+c_2\DD xu +c_3\Dn x3u +c_4\Dn x4u \right].
\end{equation*}
\cref{sechochwave1} also codes provision for a microscale \emph{lattice} version of the advection, namely \(u_t=-c\big[u(t,x+d)-u(t,x-d)\big]/(2dH)\)\,.
The computer analysis in both cases learns discrete macroscale models with high-order consistency to the corresponding given perturbed wave system.

\begin{remark}\label{remcom}
The consistency results of \cref{secpcmwave1,Xsecpchdd} use deductions expressed in the operator algebra of spatial operators \(\partial_x\), \(\delta_x\) and~\(\mu_x\).
The application of these operators evaluates fields at specific spatial points, but exactly how depends on the scale at which they are operating. 
At the microscale~\(u_j\) is governed by a given \pde{} (such as~\eqref{eqgenpde} in the general case) and this \pde\ justifies the evaluation of such operators.
For example, in the context of the wave \pde, \(\partial_tu_j=-c\partial_xu_j\)\,, effectively \(\partial_x=-\tfrac1c\partial_t\) and \(\delta_x=2\sinh[H\partial_x/2]=-2\sinh[H\partial_t/(2c)]\)  (\cref{tblopids}). 
Thus such spatial operators can be interpreted as transformed to time derivatives via the \pde, evaluated at a point, and then transformed back via the \pde\ to spatial operators.
Often, when operating at the microscale, we convert to the subgrid variable \(\xi=(x-L_i)/H\) with \(\partial_{\xi}=\partial_x/H\)\,, \(\delta_\xi=\delta_x\) and \(\mu_\xi=\mu_x\)\,.
In contrast, the macroscale grid fields~\(U_j\) are discrete so there is no formal spatial derivative, but the main aim of this article is to construct an approximate macroscale \pde\ from a discretised evolution equation of macroscale grid fields on elements~\(j\), and to this end we define the spatial derivative at the macroscale in terms of finite-difference equations of element indices~\(j\). 
For example,  to obtain~\eqref{eqepdewave1} from~\eqref{eqgopwave1} we substitute \(\delta_j=\delta_x=2\sinh(H\partial_x/2)\)\,.
\end{remark}

\subsection{Prove consistency to the first-order wave PDE}
\label{secpcmwave1}

As indicated by its equivalent \pde~\eqref{eqepdewave1}, we here prove that the holistic discretisation~\cref{eqgopwave1} is consistent to the first-order wave \pde~\cref{equniwpde}, for every order of analysis. 
\cref{lemcc1eoca,lemwavemed} establish two key identities for the subsequent consistency \cref{thmhdcuwpde,thmfocwave},  but first we prove coupling identities for derivatives of the subgrid fields which apply in \text{subsequent lemmas.}

\begin{lemma}
\label{lemwaveds}
Let the subgrid fields~\(u_j\) be smooth, satisfy the wave \pde~\cref{equniwpde}, and be coupled by~\cref{eqiecca}.
Then the edge conditions~\cref{eqiecca} hold for all spatial derivatives of~\(u_j\).
\end{lemma}

\begin{proof}
Denote the \(n\)th derivative \(u_j^{(n)}:=\partial_x^nu_j\)\,.
Proceed via induction. 
For every \(n\geq1\) consider, recalling superscripts~\(R,L\) denote evaluation at the right\slash left edge of an element,
\begin{align*}
&c(1-\tfrac12\gamma)(u_j^{(n)R}-u_j^{(n)L})-\tfrac c2\gamma(u_{j+1}^{(n)L}-u_{j-1}^{(n)R})
\\&{}-\tfrac c2\gamma\theta(u_j^{(n)L}+u_j^{(n)R})+\tfrac c2\gamma\theta(u_{j+1}^{(n)L}+u_{j-1}^{(n)R})
\\&\quad(\text{due to the identity }u_j^{(n)}=\partial_xu_j^{(n-1)})
\\&=c(1-\tfrac12\gamma)(\partial_xu_j^{(n-1)R}-\partial_xu_j^{(n-1)L})-\tfrac c2\gamma(\partial_xu_{j+1}^{(n-1)L}-\partial_xu_{j-1}^{(n-1)R})
\\&{}-\tfrac c2\gamma\theta(\partial_xu_j^{(n-1)L}+\partial_x^2u_j^{(n-1)R})+\tfrac c2\gamma\theta(\partial_xu_{j+1}^{(n-1)L}+\partial_xu_{j-1}^{(n-1)R})
\\&\quad(\text{then using the \pde\ as in \cref{remcom}})
\\&=-(1-\tfrac12\gamma)(\partial_tu_j^{(n-1)R}-\partial_tu_j^{(n-1)L})+\tfrac12\gamma(\partial_tu_{j+1}^{(n-1)L}-\partial_tu_{j-1}^{(n-1)R})
\\&{}+\tfrac12\gamma\theta(\partial_tu_j^{(n-1)L}+\partial_tu_j^{(n-1)R})-\tfrac12\gamma\theta(\partial_tu_{j+1}^{(n-1)L}+\partial_tu_{j-1}^{(n-1)R})
\\&=-\partial_t\left\{(1-\tfrac12\gamma)(u_j^{(n-1)R}-u_j^{(n-1)L})-\tfrac12\gamma(u_{j+1}^{(n-1)L}-u_{j-1}^{(n-1)R})\right.
\\&{}\left.-\tfrac12\gamma\theta(u_j^{(n-1)L}+u_j^{(n-1)R})+\tfrac12\gamma\theta(u_{j+1}^{(n-1)L}+u_{j-1}^{(n-1)R})\right\}
\\&=\partial_t\left\{0\right\}=0\,,
\end{align*}
where the penultimate equality above holds provided \(u^{(n-1)}_j\) satisfies~\cref{eqiecca}.
By induction, since \(u^{(0)}\) satisfies~\cref{eqiecca} then so does~\(u^{(n)}\) for every~\(n\geq0\)\,.
\end{proof}

Now we recast the coupling condition so the spatial difference is on the left-hand side and all coupling terms are on the right-hand side.

\begin{lemma}\label{lemcc1eoca}
Let all elements be of length \(H_j=H\)\,. 
Define the mid-element locations \(X_j:=(R_j+L_j)/2\)\,.
Then the inter-element edge condition~\cref{eqiecca} is equivalent to the coupling condition
\begin{equation}
\delta_x u_j
=\frac{\gamma(\mu_j-\tfrac12\theta\delta_j)\delta_j}{\sqrt{1+\tfrac12\gamma\delta_j^2-\gamma\theta\mu_j\delta_j+\tfrac14\gamma^2(\theta^2-1)\delta_j^2}}u_j
\quad\text{at }x=X_j\,.
\label{eqieccat}
\end{equation}
Further, this identity also holds for every spatial derivative \(\partial_x^n u_j\) evaluated at the mid-element point \(x=X_j\)\,.
\end{lemma}


\begin{proof}
We first rewrite the edge conditions~\cref{eqiecca} in terms of mean and difference operators \(\mu\)~and~\(\delta\).
Identities herein involving~\(u_j\) are evaluated ``at \(x=X_j\)''---equivalently ``at \(\xi=1/2\)'' in terms of the subgrid space variable \(\xi:=(x-L_j)/H\) (as introduced in the subgrid fields~\eqref{eqsdiffpde})  because in this proof we regard~\(u_j\) as a function of~\(\xi\), not directly~\(x\).
Using the spatial shift operators~\(E_j\) and~\(E_\xi\) (\cref{tblopids}), \(u_j^R=E_\xi^{1/2}u_j\)\,, \(u_j^L=E_\xi^{-1/2}u_j\)\,, \(u_{j-1}^R=E_j^{-1}E_\xi^{1/2}u_j\)\,, \(u_{j+1}^L=E_jE_\xi^{-1/2}u_j\)\,, which we substitute into~\eqref{eqiecca}, and then replace the shift operators with mean and difference operators to obtain after some rearrangement that~\eqref{eqiecca} is the same as
\begin{equation}
(1+\tfrac14\gamma\delta_j^2-\tfrac12\gamma\theta\mu_j\delta_j)\delta_\xi u_j=\gamma(\mu_j-\tfrac12\theta\delta_j)\delta_j\mu_\xi u_j\,.\label{eq:ccshift}
\end{equation}
On squaring the operators\footnote{
Equation~\cref{eq:ccshift} is an operation on~\(u_j\) evaluated at \(\xi=\frac12\) to define the coupling conditions at element edges \(\xi=0,1\)\,.
When squaring the operators we must smoothly extend \(u_j\) over \(\xi\in[-\frac12,0),(1,\frac32]\)  (as defined by the \pde{}) with the new `squared operator' equation defining new coupling conditions at \(\xi=-\frac12,\frac32\) (or equivalently, at \(x=X_{j-1},X_{j+1}\) but still in the smooth extension of element~\(j\)).
For example, \(\delta^2_{\xi} u_j(\xi=\frac12)=u_j(\frac32)-u_j(-\frac12)\) and we remain in a smooth extension \text{of element~\(j\).}
}
\begin{align*}
(1+\tfrac14\gamma\delta_j^2-\tfrac12\gamma\theta\mu_j\delta_j)^2\delta_\xi ^2 u_j=\gamma^2(\mu_j-\tfrac12\theta\delta_j)^2\delta_j^2\mu_{\xi}^2 u_j\,.
\end{align*}
Expand both sides, and on the left substitute \(\mu_j^2=1+\frac14\delta_j^2\) while on the right substitute \(\mu_\xi^2=1+\frac14\delta_\xi^2\) (\cref{tblopids}), to obtain
\begin{equation*}
\big[1+\tfrac12\gamma\delta_j^2-\gamma\theta\mu_j\delta_j+\tfrac14\gamma^2(\theta^2-1)\delta_j^2\big]\delta_\xi^2 u_j
={\gamma^2(\mu_j^2-\tfrac14\theta^2\delta_j^2)^2\delta_j^2}u_j\,.
\end{equation*}
Since \(\delta_\xi=\delta_x\) on the microscale, dividing the above by the~\([\cdot]\) factor, taking the square root, and recalling that this operator applies to~\(u_j\) at \(\xi=\tfrac12\) (that is, \(x=X_j\)), \text{we obtain~\eqref{eqieccat}.}
\end{proof}

\begin{lemma}\label{lemwavemed}
The first-order wave \pde~\cref{equniwpde} on elements with coupling~\cref{eqieccat} has macroscale  dynamics such that
\begin{equation}
\partial_tU_j=-\frac{2c}H\asinh\left[ \frac{\tfrac12\gamma(\mu_j-\tfrac12\theta\delta_j)\delta_j}{\sqrt{1+\tfrac12\gamma\delta_j^2-\gamma\theta\mu_j\delta_j+\tfrac14\gamma^2(\theta^2-1)\delta_j^2}} \right] U_j\,,
\label{eqeqpdewave1}
\end{equation}
where macroscale grid field~\(U_j\) is a local average of~\(u_j(x)\) about the centre of the \(j\)th element; that is, \(U_j:=\tfrac{1}{2\ell}\int_{X_j-\ell}^{X_j+\ell}w(x-X_j)u_j(x)\,dx\) for some \(\ell\leq H/2\) and some weight function~\(w\) (the macroscale field reduces to the mid-element value \(U_j:=u_j(X_j)\) as \(\ell\rightarrow 0\) and for weight \(w(0)=1\)).
\end{lemma}

Knowing how spatial differences of \pde\ solutions transform now translates into the corresponding evolution of the macroscale variables.

\begin{proof} 
Consider the wave \pde~\cref{equniwpde} in the \(j\)th~element and higher order spatial derivatives (and using \cref{tblopids}):  for every \(p=0,1,2,\ldots\), \(\partial_t\partial_x^p u_j=-c\partial_x \partial_x^p u_j=-\tfrac cH2\asinh(\delta_x/2) \partial_x^pu_j\)\,.
Inverting the operator~\(\asinh(\delta_x/2)\), this identity is equivalent to \(2\sinh(-H\partial_t/2c) \partial_x^p u_j=\delta_x \partial_x^p u_j\)\,.
Evaluating this at the element mid-point and by the coupling~\cref{eqieccat}, 
\begin{align*}
2\sinh(-H\partial_t/2c) \partial_x^p u_j\big|_{X_j}
&=\delta_x \partial_x^p u_j\big|_{X_j}\\
&=\frac{\gamma(\mu_j-\tfrac12\theta\delta_j)\delta_j}{\sqrt{1+\tfrac12\gamma\delta_j^2-\gamma\theta\mu_j\delta_j+\tfrac14\gamma^2(\theta^2-1)\delta_j^2}} \partial_x^p u_j\big|_{X_j}\,.
\end{align*}
Premultiplying by \((x-X_j)^p/p!\) provides the \(p\)th term in a Taylor expansion of~\(u_j(x)\) about~\(x=X_j\) (the term~\((x-X_j)^p=H^p(\xi-\frac12)^p\) commutes with~\(\partial_t\), \(\delta_j\)~and~\(\mu_j\) since it is independent of \(t\)~and~\(j\), but it does not commute with~\(\delta_x=\delta_{\xi}\) when \(p>0\)) 
and on summing over all~\(p=0,1,2,\ldots\) we obtain
\begin{equation}
2\sinh(-H\partial_t/2c) u_j(x)=\frac{\gamma(\mu_j-\tfrac12\theta\delta_j)\delta_j}{\sqrt{1+\tfrac12\gamma\delta_j^2-\gamma\theta\mu_j\delta_j+\tfrac14\gamma^2(\theta^2-1)\delta_j^2}} u_j(x)\label{eqieccatAllx}
\end{equation}
in some open interval containing~\(X_j\).
Then multiply by the weight function~\(w(x-X_j)\) and integrate over the centre of the \(j\)th~element from \(x=X_j-\ell\) to~\(X_j+\ell\) and finally
revert the \(\sinh(-H\partial_t/2c)\) operator to deduce~\cref{eqeqpdewave1}; or, for the averaging width \(\ell\rightarrow 0\)\,, set \(x=X_j\) and revert the \(\sinh\) operator to deduce~\cref{eqeqpdewave1}.
\end{proof}

The above proof also shows that~\cref{eqieccat} holds for \(u_j\) evaluated at all \(x\in(X_j-\ell,X_j+\ell)\)\,, not just \(x=X_j\)\,. 
To prove this, in~\cref{eqieccatAllx} substitute \(2\sinh(-H\partial_t/2c)=2\sinh(H\partial_x/2)=\delta_x\)\,.

Computer algebra readily learns the dynamics~\eqref{eqeqpdewave1} in the form of a Taylor series of  the coupling parameter~\(\gamma\). 
That is, the computer algebra recovers precisely the series~\cref{eqgopwave1} to a chosen order of~\(\gamma\).
This recovery of the series verifies \cref{lemwaveds,lemcc1eoca,lemwavemed}.

\begin{theorem}\label{thmhdcuwpde}
At full coupling \(\gamma=1\)\,, and for every~\(\theta\), the holistic discretisation~\eqref{eqeqpdewave1}, equivalently~\eqref{eqgopwave1}, is consistent with the first-order wave \pde~\eqref{equniwpde}.
\end{theorem}
\begin{proof} 
Evaluating~\cref{eqeqpdewave1} at full-coupling \(\gamma=1\) gives (\cref{tblopids})
\begin{align*}
\partial_tU_j&=-\frac{2c}H\asinh\left[ \frac{\tfrac12(\mu_j-\tfrac12\theta\delta_j)\delta_j}{\sqrt{1+\tfrac12\delta_j^2-\theta\mu_j\delta_j+\tfrac14(\theta^2-1)\delta_j^2}} \right] U_j
\\&=-\frac{2c}H\asinh\left[\frac{\tfrac12(\mu_j-\tfrac12\theta\delta_j)\delta_j}{\sqrt{\mu_j^2-\theta\mu_j\delta_j+\tfrac14\theta^2\delta_j^2}} \right] U_j
\\&=-\frac{2c}H\asinh\left[ \frac{\tfrac12(\mu_j-\tfrac12\theta\delta_j)\delta_j}{\sqrt{(\mu_j-\tfrac12\theta\delta_j)^2}}
 \right] U_j
 \\&=-\frac{2c}H\asinh\left[  \delta_j/2\right] U_j=-\frac{2c}H\asinh\left[  \delta_x/2\right] U_j\quad\text{(\cref{remcom})}
\\&=-\frac cH\partial_j U_j=-c\partial_x U_j\,.
\end{align*}
That \(\partial_tU_j=-c\partial_xU_j\) confirms that computing on elements with full coupling recovers macroscale simulations and predictions consistent with the macroscale dynamics of the first-order wave \pde~\cref{equniwpde}.
\end{proof}

Substituting \(\gamma=1\) reduces the series~\cref{eqgopwave1}, obtained from computer algebra, to the  first-order wave \pde~\cref{equniwpde} and verifies \cref{thmhdcuwpde}.
However, a practical derivation of the series~\cref{eqgopwave1} necessarily truncates to some finite order in the coupling, say to errors~\ord{\gamma^p}. 
Two important questions concerning this error in~\(\gamma\) and the practical application of this holistic discretisation are: 
how is this error reinterpreted in terms of a stencil of nearest-element coupling? 
and how does it affect the consistency of the series~\cref{eqgopwave1} with the wave \pde~\cref{equniwpde}?

\begin{theorem}\label{thmfocwave}
Consider constructing, to errors~\ord{\gamma^p}, the slow manifold of the first-order wave \pde~\cref{equniwpde} on elements coupled by~\cref{eqiecca}. 
The resulting holistic discretisation (that is, \cref{eqgopwave1} truncated to errors~\ord{\gamma^p}) is \begin{itemize}
\item a scheme of stencil width \((2p-1)\) on the macroscale grid, and
\item consistent with the wave \pde\ to errors~\ord{\partial_x^p}.
\end{itemize}
\end{theorem}

\begin{proof} 
Consider macroscale solutions for which the differences~\(\delta_j\) are `small' (i.e., \(U_j\) varies slowly with changes in~\(j\)).
For such solutions, the Taylor series in~\(\delta_j\) of the holistic discretisation~\cref{eqeqpdewave1} is appropriate.
But we only need the pattern of powers of~\(\gamma\), \(\mu_j\) and~\(\delta_j\), so in the derivation we \emph{omit almost all details of the coefficients, and use~\(\sim\) instead of~\(=\) to denote such omission}.

The \(\asinh\) term  in~\cref{eqeqpdewave1} is
\begin{equation}
\asinh\left[ \frac{\tfrac12(\mu_j-\tfrac12\theta\delta_j)(\gamma\delta_j)}{\sqrt{1+\tfrac12\delta_j(\gamma\delta_j)-\theta\mu_j(\gamma\delta_j)+\tfrac14(\theta^2-1)(\gamma\delta_j)^2}} \right],
\end{equation}
where we pair each instance of~\(\gamma\) with an instance of~\(\delta_j\).
As all \(\gamma\) are paired with a \(\delta_j\), but not all \(\delta_j\) (or \(\mu_j\)) are paired with a \(\gamma\), we conclude that in a Taylor expansion in small \(\delta_j\), up to a given power~\(\delta_j^p\), every~\(\gamma\) exponent is~\(\leq p\) (as in~\cref{eqgopwave1}).
Furthermore, truncating construction to errors~\ord{\gamma^p} ensures that the terms are all complete and correct for every power of~\(\delta_j\) less than~\(p\); that is, the error is~\ord{\delta_j^p}.
Since \(\partial_x\sim\delta_j\) when operating on macroscale grid fields,  in such a construction the consistency error is~\ord{\partial_x^p} as asserted \text{in the theorem.}

To determine the width of the stencil it is sufficient to expand
\begin{align*}
\asinh[\cdot]&\sim\asinh\left[\frac{(\mu_j+\delta_j)(\gamma\delta_j)}{\sqrt{1+(\mu_j+\delta_j)(\gamma\delta_j)+(\gamma\delta_j)^2}}\right]
\\&\sim \sum_{m=0}^{\infty}(\mu_j+\delta_j)^{2m+1}(\gamma\delta_j)^{2m+1}[1+(\mu_j+\delta_j)(\gamma\delta_j)+(\gamma\delta_j)^2]^{-m-1/2}
\\&\sim \sum_{m=0}^{\infty}(\mu_j+\delta_j)^{2m+1}(\gamma\delta_j)^{2m+1}\sum_{n=0}^{\infty}(\mu_j+\delta_j)^{0:n}(\gamma\delta_j)^{n:2n} 
\\&\sim \sum_{m,n=0}^{\infty}(\mu_j+\delta_j)^{2m+1:2m+n+1}(\gamma\delta_j)^{2m+n+1:2m+2n+1}\,.
\end{align*}
On expanding to errors~\ord{\gamma^p} we retain up to \(\gamma^{p-1}\).
Then, the lowest order~\(\gamma\) which obtains the highest order \(E_j^{\pm1/2}\sim \mu_j,\delta_j\) is when \(p-1=2m+n+1\)\,, so that in the above expansion we have up to \(E_j^{[(2m+n+1)+(2m+n+1)]/2}\sim E_j^{\pm (p-1)}\)\,.
So the stencil is of width~\((2p-1)\). 
\end{proof}

This completes the proof of the consistency of these fixed-bandwidth holistic discretisations, elements coupled by~\eqref{eqiecca}, to the first-order wave \pde~\cref{equniwpde}.

\subsection{Holistic discretisation of generalised diffusion is consistent}
\label{Xsecpchdd}

Recall that \cref{eghocdd} shows that in the case of the diffusion \pde, \(u_t=u_{xx}\)\,, the holistic discretisation~\cref{eqgopdiff} is consistent with the \pde\ to all computed orders of accuracy.
The computer algebra (\cref{sechochd}) deriving the results~\cref{eqsdiffpde} is easily modified to analyse more general \pde{}s (such as advection-diffusion, \(u_t=-cu_x+u_{xx}\)\,, or lattice diffusion, \(u_t=u|_{x+d}-2u+u|_{x-d}\)) and in all cases investigated we found corresponding high-order consistency.
Consequently, here we prove consistency for a generalised class of \pde{}s.
The proofs in this section are analogous to those in \cref{secpcmwave1} for wave \pde{}s.

The proof of consistency is cognate to an earlier proof \cite[\S6.1]{Roberts2011a} of high-order consistency emerging from a distinctly different inter-element coupling.
However, the new inter-element coupling~\eqref{eqsiecc},  implemented on element edges in order to preserve important symmetries, requires \text{new proof.}

The analysis here  explores, for some Hilbert space~\UU, the evolution in time~\(t\) of a \UU-valued field~\(u(t,x)\) over space~\(x\in\XX\)\,.
We restrict attention to smooth \UU-valued fields~\(u\) in the vector space \(\CC:=C^\omega(\XX)\)\,, and correspondingly \(\CC_j:=C^\omega(\XX_j)\) on the spatial elements.
The field~\(u\) is to satisfy the isotropic, homogenous, generalised diffusion-like \pde, the evolution equation,
\begin{equation}
u_t=\cK_0 u+\cK(\partial_x^2)u\,,
\label{Xeqdiffpde}
\end{equation}
for some linear operators \(\cK_k:\UU\to\UU\) and isotropic invertible \(\cK:\CC\to \CC\)\,, with inverse~\(\cK^{-1}\), where \(\cK\)~is formally expandable as \(\cK=\sum_{k=1}^\infty \cK_k\partial_x^{2k}\) with \(\cK_1\)~invertible.

The most basic example of~\eqref{Xeqdiffpde} is simple diffusion, \(u_t=u_{xx}\)\,, obtained with \(\UU=\RR\)\,, zero~\(\cK_0\), and identity~\cK.
For a second example, with microscale space step~\(d\) the difference equation \(u_t=\delta^2 u:=u|_{x+d}-2u+u|_{x-d}\) corresponds to (\cref{tblopids}) setting \(\cK_0:=0\) and \(\cK(\partial_x^2) :=4\sinh^2(d\partial_x/2) =d^2\partial_x^2 +\tfrac1{12}d^4\partial_x^4 +\tfrac1{360}d^6\partial_x^6 +\cdots\) and so \(\cK_1=d^2\)\,.
Further in this second example, the invertibility gives that since \(\delta^2=\cK(\partial_x^2)\) then \(\partial_x^2 =\cK^{-1}(\delta^2) =(4/d^2)\asinh^2(\delta/2) =\tfrac1{d^2}\big[\delta^2 -\tfrac1{12}\delta^4 +\tfrac1{90}\delta^6-\cdots \big]\).

\cref{Xlemrpde,Xlemhods,Xlemcc1eocb,Xlemdiffmed} establish key identities for the subsequent consistency \cref{Xthmhdcudiff,Xthmfocdiff}.

\begin{lemma}\label{Xlemrpde}
Define the transformed time-derivative operator \(\cT:=\cK^{-1}(\partial_t-\cK_0)\)\,.
Then the evolution equation~\cref{Xeqdiffpde} is equivalent to
\begin{equation}
\cT u=u_{xx}\,.
\label{Xeqrpde}
\end{equation}
\end{lemma}
\begin{proof}
As in \cref{remcom}, and since~\cK\ is invertible, we rearrange the operator equation \(\partial_t=\cK_0+\cK(\partial_x^2)\) to \(\cK^{-1}(\partial_t-\cK_0)=\partial_x^2\)\,. 
Hence, \eqref{Xeqrpde}~is equivalent to the evolution equation~\eqref{Xeqdiffpde}.
\end{proof}

\begin{lemma}
\label{Xlemhods}
Let the subgrid fields~\(u_j\in\CC_j\) satisfy the evolution equation~\cref{Xeqdiffpde}, equivalently~\eqref{Xeqrpde}, and be coupled by edge conditions~\cref{eqsiecc} with flux \(f_j:=- u_{jx}\)\,.
Then the edge condition~\cref{eqiecca} holds for all even spatial derivatives of~\(u_j\), and the edge condition~\cref{eqieccb} holds for all odd spatial derivatives of~\(u_j\).
\end{lemma}

\begin{proof}
Denote the \(n\)th derivative \(u_j^{(n)}:=\partial_x^nu_j\)\,.
Proceed via induction. 
From the definition of the flux, multiplying the flux edge condition~\cref{eqieccb} by~\(-1\) is the edge condition for the first derivative of~\(u_j\), and this coupling is equivalent to~\cref{eqiecca} with \(\theta\mapsto -\theta\)\,.
Then for every \(n\geq2\) consider
\begin{align*}
&(1-\tfrac12\gamma)(u_j^{(n)R}-u_j^{(n)L})-\tfrac12\gamma(u_{j+1}^{(n)L}-u_{j-1}^{(n)R})
\\&\quad{}
-\tfrac12\gamma(-1)^n\theta(u_j^{(n)L}+u_j^{(n)R})
+\tfrac12\gamma(-1)^n\theta(u_{j+1}^{(n)L}+u_{j-1}^{(n)R})
\\&\quad(\text{due to the identity }u_j^{(n)}=\partial_x^2u_j^{(n-2)})
\\&=(1-\tfrac12\gamma)(\partial_x^2u_j^{(n-2)R}-\partial_x^2u_j^{(n-2)L})
-\tfrac12\gamma(\partial_x^2u_{j+1}^{(n-2)L}-\partial_x^2u_{j-1}^{(n-2)R})
\\&\quad{}
-\tfrac12\gamma(-1)^n\theta(\partial_x^2u_j^{(n-2)L}+\partial_x^2u_j^{(n-2)R})
+\tfrac12\gamma(-1)^n\theta(\partial_x^2u_{j+1}^{(n-2)L}+\partial_x^2u_{j-1}^{(n-2)R})
\\&\quad(\text{then using \eqref{Xeqrpde}})
\\&=(1-\tfrac12\gamma)(\cT u_j^{(n-2)R}-\cT u_j^{(n-2)L})
-\tfrac12\gamma(\cT u_{j+1}^{(n-2)L}-\cT u_{j-1}^{(n-2)R})
\\&\quad{}
-\tfrac12\gamma(-1)^n\theta(\cT u_j^{(n-2)L}+\cT u_j^{(n-2)R})
+\tfrac12\gamma(-1)^n\theta(\cT u_{j+1}^{(n-2)L}+\cT u_{j-1}^{(n-2)R})
\\&=\cT \left\{(1-\tfrac12\gamma)(u_j^{(n-2)R}-u_j^{(n-2)L})
-\tfrac12\gamma(u_{j+1}^{(n-2)L}-u_{j-1}^{(n-2)R})
\right.\\&\left.\quad{}
-\tfrac12\gamma(-1)^n\theta(u_j^{(n-2)L}+u_j^{(n-2)R})
+\tfrac12\gamma(-1)^n\theta(u_{j+1}^{(n-2)L}+u_{j-1}^{(n-2)R})\right\}
\\&=\cT \left\{0\right\}=0\,,
\end{align*}
where the penultimate equality above holds provided \(u^{(n-2)}_j\) satisfies~\cref{eqiecca} when \(n\)~is even and~\cref{eqieccb} when \(n\)~is odd.
By induction, since \(u^{(0)}\) satisfies~\cref{eqiecca} then so does~\(u^{(n)}\) for every even~\(n\geq0\)\,, and 
since \(u^{(1)}\) satisfies~\cref{eqieccb} then so does~\(u^{(n)}\) for every odd~\(n\geq0\)\,.
\end{proof}

In the case of simple diffusion, the code of \cref{sechochd} verifies using computer algebra that for the constructed slow manifold discretisation (\cref{seclin}), the field~\(u_j(t,x)\) satisfies the following identity~\cref{Xeqieccatdif}.  
We now prove this identity more generally.
The proof is similar to that of \cref{lemcc1eoca}, but a little more complex as we now have edge conditions for both field and flux.

\begin{lemma}\label{Xlemcc1eocb}
Recall the mid-element locations \(X_j:=(R_j+L_j)/2\)\,.
For~\(u_j(t,x)\) as specified in \cref{Xlemhods}, and \(u_j\)~smoothly extended to being over~\([X_{j-1},X_{j+1}]\),
\begin{equation}
\delta_x^2 u_j=\frac{\gamma^2(\mu_j^2-\tfrac14\theta^2\delta_j^2)\delta_j^2}{1+\tfrac12\gamma\delta_j^2-\tfrac14\gamma^2(\theta^2+1)\delta_j^2}u_j
\quad\text{at }x=X_j\,.
\label{Xeqieccatdif}
\end{equation}
Further, this identity also holds for every spatial derivative~\(\partial_x^nu_j\).
\end{lemma}

\begin{proof} 
The field and flux edge conditions~\cref{eqsiecc}, and their derivatives (\cref{Xlemhods}), imply that,  in terms of mean and difference operators, the following holds at \(x=X_j\)\,, for every integer \(k=0,1,\ldots\)\,,
\begin{align}
[\delta_x+\gamma(\tfrac14\delta_j\delta_x-\mu_j\mu_x)\delta_j]\partial_x^{2k}u_j 
\quad&=\phantom{+}\tfrac12\gamma\theta(\mu_j\delta_x-\delta_j\mu_x) \delta_j\partial_x^{2k} u_j\,,
\nonumber\\
[\delta_x+\gamma(\tfrac14\delta_j\delta_x-\mu_j\mu_x)\delta_j]\partial_x^{2k+1}u_j 
&=-\tfrac12\gamma\theta(\mu_j\delta_x-\delta_j\mu_x) \delta_j\partial_x^{2k+1} u_j\,.
\label{Xeq:cck}
\end{align}
Upon multiplying by appropriate factors, and changing to subgrid variable \(\xi=(x-L_j)/H\)\,, the above set of edge conditions is equivalent to the following set: at \(\xi=\tfrac12\)\,, for every \(k=0,1,\ldots\)
\begin{subequations}\label{Xeqs:cckxi}%
\begin{align}
\frac{[\delta_\xi +\gamma(\tfrac14\delta_j\delta_\xi -\mu_j\mu_\xi )\delta_j]\ \partial_\xi ^{2k}u_j}{2^{2k}(2k)!} 
\quad&=\phantom{+}\frac{\tfrac12\gamma\theta(\mu_j\delta_\xi -\delta_j\mu_\xi ) \delta_j\ \partial_\xi ^{2k}u_j}{2^{2k}(2k)!} \,,
\label{Xeq:ccka}
\\
\frac{[\delta_\xi +\gamma(\tfrac14\delta_j\delta_\xi -\mu_j\mu_\xi )\delta_j]\ \partial_\xi ^{2k+1}u_j}{2^{2k+1}(2k+1)!}
&=-\frac{\tfrac12\gamma\theta(\mu_j\delta_\xi -\delta_j\mu_\xi ) \delta_j\ \partial_\xi ^{2k+1}u_j}{2^{2k+1}(2k+1)!} \,.
\label{Xeq:cckb}
\end{align}
\end{subequations}
Summing each of~\eqref{Xeqs:cckxi} over \(k=0,1,2,\ldots\)\,, and for simplicity here omitting ``\(u_j\) at \(\xi=\tfrac12\)'', gives
\begin{align*}
[\delta_\xi +\gamma(\tfrac14\delta_j\delta_\xi -\mu_j\mu_\xi )\delta_j]\cosh(\tfrac{1}{2}\partial_\xi )
&=\phantom{+}\tfrac12\gamma\theta(\mu_j\delta_\xi -\delta_j\mu_\xi ) \delta_j\cosh(\tfrac{1}{2}\partial_\xi ),\\
[\delta_\xi +\gamma(\tfrac14\delta_j\delta_\xi -\mu_j\mu_\xi )\delta_j]\sinh(\tfrac{1}{2}\partial_\xi )
&=-\tfrac12\gamma\theta(\mu_j\delta_\xi -\delta_j\mu_\xi ) \delta_j\sinh(\tfrac{1}{2}\partial_\xi ).
\end{align*}
Recall that \(\cosh(\tfrac{1}{2}\partial_\xi) =\mu_\xi\) and \(\sinh(\tfrac{1}{2}\partial_\xi) =\tfrac12\delta_\xi\)  (\cref{tblopids}).
Consequently the above pair of operator equations are
\begin{align*}
[\delta_\xi +\gamma(\tfrac14\delta_j\delta_\xi -\mu_j\mu_\xi )\delta_j]\mu_\xi 
&=\phantom{+}\tfrac12\gamma\theta(\mu_j\delta_\xi -\delta_j\mu_\xi ) \delta_j\mu_\xi \,,\nonumber\\
[\delta_\xi +\gamma(\tfrac14\delta_j\delta_\xi -\mu_j\mu_\xi )\delta_j]\delta_\xi 
&=-\tfrac12\gamma\theta(\mu_j\delta_\xi -\delta_j\mu_\xi ) \delta_j\delta_\xi \,.
\end{align*}
Each of these equations rearrange, respectively, to
\begin{align*}
(1+\tfrac14\gamma\delta_j^2-\tfrac12\gamma\theta\mu_j\delta_j)\mu_\xi \delta_\xi  
&=\gamma(\mu_j-\tfrac12\theta\delta_j)\delta_j\mu_\xi ^2  
\,,\\
(1+\tfrac14\gamma\delta_j^2+\tfrac12\gamma\theta\mu_j\delta_j)\delta_\xi ^2
\quad&= \gamma(\mu_j+\tfrac12\theta\delta_j)\delta_j\mu_\xi \delta_\xi \,.
\end{align*}
Multiply both sides of the first equation by \(\gamma(\mu_j+\tfrac12\theta\delta_j)\delta_j\)\,, then substitute the second to obtain\footnote{Here we smoothly extend the \(j\)th~element (\(\xi\in[0,1]\)) to also include \(\xi\in[-\frac12,0),(1,\frac32]\) and the following equation is an additional coupling condition for element~\(j\) at \(\xi=-\frac12, \frac32\)\,, or equivalently \(x=X_{j-1},X_{j+1}\) (as in the proof of \cref{lemcc1eoca}).}
\begin{align*}
&(1+\tfrac14\gamma\delta_j^2-\tfrac12\gamma\theta\mu_j\delta_j)(1+\tfrac14\gamma\delta_j^2+\tfrac12\gamma\theta\mu_j\delta_j)\delta_\xi ^2 \\
&=\gamma^2(\mu_j+\tfrac12\theta\delta_j)(\mu_j-\tfrac12\theta\delta_j)\delta_j^2 \,.
\end{align*}
Expand both sides, and on the left substitute \(\mu_j^2=1+\frac14\delta_j^2\) whereas on the right substitute \(\mu_\xi^2=1+\frac14\delta_\xi^2\) (\cref{tblopids}), to obtain
\begin{equation*}
\big[1+\tfrac12\gamma\delta_j^2-\tfrac14\gamma^2(\theta^2+1)\delta_j^2\big]\delta_\xi^2
={\gamma^2(\mu_j^2-\tfrac14\theta^2\delta_j^2)\delta_j^2}\,.
\end{equation*}
Since \(\delta_\xi=\delta_x\)\,, dividing the above by the~\([\cdot]\) factor, and now explicitly applying to~\(u_j\) at \(\xi=\tfrac12\) (that is, \(x=X_j\)) gives~\eqref{Xeqieccatdif}.

Furthermore, to prove that~\eqref{Xeqieccatdif} holds for all spatial derivatives~\(\partial_x^nu_j\), adapt~\eqref{Xeq:cck}
for \(n\)~even, so that \(n+2k\) is even:
\begin{align*}
[\delta_x+\gamma(\tfrac14\delta_j\delta_x-\mu_j\mu_x)\delta_j]\partial_x^{2k}\partial_x^nu_j 
\quad&=\phantom{+}\tfrac12\gamma\theta(\mu_j\delta_x-\delta_j\mu_x) \delta_j\partial_x^{2k}\partial_x^n u_j\,,\nonumber\\
[\delta_x+\gamma(\tfrac14\delta_j\delta_x-\mu_j\mu_x)\delta_j]\partial_x^{2k+1}\partial_x^nu_j 
&=-\tfrac12\gamma\theta(\mu_j\delta_x-\delta_j\mu_x) \delta_j\partial_x^{2k+1}\partial_x^nu_j\,,
\end{align*}
and for \(n\)~odd, so that \(n+2k\) is odd:
\begin{align*}
[\delta_x+\gamma(\tfrac14\delta_j\delta_x-\mu_j\mu_x)\delta_j]\partial_x^{2k}\partial_x^nu_j 
\quad&=-\tfrac12\gamma\theta(\mu_j\delta_x-\delta_j\mu_x) \delta_j\partial_x^{2k}\partial_x^n\,,\nonumber\\
[\delta_x+\gamma(\tfrac14\delta_j\delta_x-\mu_j\mu_x)\delta_j]\partial_x^{2k+1}\partial_x^nu_j 
&=\phantom{+}\tfrac12\gamma\theta(\mu_j\delta_x-\delta_j\mu_x) \delta_j\partial_x^{2k+1}\partial_x^nu_j\,.
\end{align*}
Then, in the proof of~\eqref{Xeqieccatdif} for~\(u_j\) we replace \(u_j\)~with~\(\partial_x^nu_j\) and \(\theta\)~with~\((-1)^n\theta\).
Since only~\(\theta^2\) terms arise in~\eqref{Xeqieccatdif}, it thus also holds for~\(\partial_x^nu_j\).
\end{proof}

\begin{lemma}\label{Xlemdiffmed}
The evolution equation~\eqref{Xeqdiffpde}, equivalently~\eqref{Xeqrpde}, on equi-sized elements for \(u_j\in\CC_j\) coupled by~\cref{eqsiecc}, has dynamics satisfying both the following two spatially discrete equations:
\begin{subequations}\label{Xeqeqpdediff}%
\begin{align}&
\cT U_j=\frac{4}{H^2}\asinh^2\left[\sqrt{\frac{\mu_j^2-\tfrac14\theta^2\delta_j^2}{1+\tfrac12\gamma\delta_j^2-\tfrac14\gamma^2(\theta^2+1)\delta_j^2}}\frac{\gamma\delta_j}{2}\right]U_j\,,
\label{Xeqeqpdediffa}\\&
\partial_tU_j=
\cK_0 U_j+\cK\left\{
\frac{4}{H^2}\asinh^2\left[\sqrt{\frac{\mu_j^2-\tfrac14\theta^2\delta_j^2}{1+\tfrac12\gamma\delta_j^2-\tfrac14\gamma^2(\theta^2+1)\delta_j^2}}\frac{\gamma\delta_j}{2}\right]
\right\}U_j\,,
\label{Xeqeqpdediffb}
\end{align}
\end{subequations}
where macroscale grid field~\(U_j(t)\in\UU\) is a local average of~\(u_j(t,x)\) about the centre of the \(j\)th element; that is, \(U_j(t):=\tfrac{1}{2\ell}\int_{X_j-\ell}^{X_j+\ell}w(x-X_j)u_j(t,x)\,dx\) for some \(\ell\leq H/2\) and some weight function~\(w\) (the macroscale field reduces to the mid-element value \(U_j(t):=u_j(t,X_j)\) as \(\ell\rightarrow 0\) and for weight \(w(0)=1\)).
\end{lemma}

\begin{proof} 
We adapt the proof of \cref{lemwavemed}.
Given~\eqref{Xeqieccatdif} of \cref{Xlemcc1eocb}, first consider the evolution equation~\eqref{Xeqrpde} in the elements (and smoothly extended from the elements): for every \(p=0,1,2,\ldots\)\,, \(\cT \partial_x^pu_j  =\partial_x^2\partial_x^pu_j =\tfrac4{H^2}\asinh^2(\delta_x/2)\partial_x^pu_j =\tfrac4{H^2}\asinh^2(\sqrt{\delta_x^2/4})u_j\)\,.
Upon inverting the operator \(\asinh^2(\sqrt{\cdot/4})\), this evolution equation is equivalent to \(4\sinh^2(H\sqrt{\cT}/2)\partial_x^p u_j=\delta_x^2\partial_x^pu_j\)\,.
Evaluating this form of the evolution equation at the element mid-point, and using~\cref{Xeqieccatdif}, gives
\begin{equation*}
4\sinh^2(H\sqrt{\cT}/2)\partial_x^p u_j\big|_{X_j}
=\delta_x^2 \partial_x^p u_j\big|_{X_j}
=\frac{\gamma^2(\mu_j^2-\tfrac14\theta^2\delta_j^2)\delta_j^2}{1+\tfrac12\gamma\delta_j^2-\tfrac14\gamma^2(\theta^2+1)\delta_j^2}\partial_x^p u_j\big|_{X_j}\,.
\end{equation*}
Premultiplying by \((x-X_j)^p/p!\) provides the \(p\)th term in a Taylor expansion of the smooth~\(u_j(t,x)\) about~\(x=X_j\) (the term~\((x-X_j)^p\) commutes with~\(\delta_t\), \(\delta_j\)~and~\(\mu_j\), but not with~\(\delta_x\) when \(p>0\)). 
Given \(u_j\in\CC_j\)\,, the Taylor series converges in some open interval containing~\(X_j\), and so summing over all~\(p\) \text{we obtain}
\begin{equation}
4\sinh^2(H\sqrt{\cT}/2)u_j(t,x)
=\frac{\gamma^2(\mu_j^2-\tfrac14\theta^2\delta_j^2)\delta_j^2}{1+\tfrac12\gamma\delta_j^2-\tfrac14\gamma^2(\theta^2+1)\delta_j^2}u_j(t,x)\label{XeqieccatdifAllx}
\end{equation}
in that interval.
Then multiply by the weight function~\(w(x-X_j)\) and integrate over the middle of the \(j\)th~element from \(x=X_j-\ell\) to~\(X_j+\ell\), and 
revert the \(\sinh(-H\cT/2c)\) operator, to deduce~\cref{Xeqeqpdediffa}; or for the averaging width \(\ell\rightarrow 0\)\,, set \(x=X_j\) and revert the \(\sinh\) operator to also deduce~\cref{Xeqeqpdediffa}.

Finally, from~\cref{Xeqeqpdediffa} revert the operator \(\cT=\cK^{-1}(\partial_t-\cK_0)\) to deduce~\eqref{Xeqeqpdediffb}.
\end{proof}

The above proof also shows that identity~\cref{Xeqieccatdif} holds for \(u_j\) evaluated at all \(x\in(X_j-\ell,X_j+\ell)\)\,, not just \(x=X_j\)\,. 
To prove this, in~\cref{XeqieccatdifAllx} substitute \(4\sinh^2(H\sqrt{\cT}/2) = \delta_x^2\) .
Similarly, the code of \cref{sechochd} shows that the identity holds for all~\(\xi\), not just \(\xi=1/2\)\,.

With these lemmas we now prove the consistency of the slow manifold discretisation.

\begin{theorem}\label{Xthmhdcudiff}
At full coupling \(\gamma=1\)\,, and for every~\(\theta\), the holistic discretisation~\eqref{Xeqeqpdediff} is consistent with the evolution equation~\eqref{Xeqdiffpde}.
\end{theorem}

For example, the holistic discretisation~\cref{eqgopdiff} is consistent with the simple diffusion \pde\ \(u_t=u_{xx}\)\,.

\begin{proof} 
Evaluating~\eqref{Xeqeqpdediffa} at full coupling \(\gamma=1\) gives (\cref{tblopids})
\begin{align*}
\cT U_j &=
\frac{4}{H^2}\asinh^2\left[\sqrt{\frac{\mu_j^2-\tfrac14\theta^2\delta_j^2}{1+\tfrac14\delta_j^2-\tfrac14\theta^2\delta_j^2}}\frac{\delta_j}{2}\right]U_j
\\&
=\frac{4}{H^2}\asinh^2\left[ {\tfrac12\delta_j} \right] U_j
=\frac1{H^2}\partial_j^2 U_j\,.
\end{align*}
Let \(U(t,x)\) be a smooth interpolation of~\(U_j(t)\), then \(\frac1H\partial_jU_j=\partial_xU\)\,, and so the above reduces to \(\cT U=\partial_x^2 U\)\,, that is, \(\cK^{-1}(\partial_t-\cK_0)U=\partial_x^2 U\)\,. 
By reverting the operator~\(\cK^{-1}\), it follows that \(\partial_tU=\cK_0 U+\cK(\partial_x^2) U\)\,.
Thus at \(\gamma=1\)\,, both equations in~\eqref{Xeqeqpdediff} are consistent to the evolution equation~\eqref{Xeqdiffpde} to all orders.  
\end{proof}

Of more practical interest is what happens when the inter-element coupling is analysed to some finite order in~\(\gamma\) in order to construct an holistic discretisation of some finite stencil width in space.

\begin{theorem}\label{Xthmfocdiff}
Consider constructing, to errors~\ord{\gamma^p}, the spatial discretisation of the evolution equation~\eqref{Xeqdiffpde} on elements coupled by~\cref{eqsiecc}. 
The resulting holistic discretisation (e.g., \cref{eqgopdiff} truncated to errors~\ord{\gamma^p}) is \begin{itemize}
\item a scheme of stencil width \((2p-1)\) on the macroscale grid, and
\item consistent with the general \pde~\eqref{Xeqdiffpde} to errors~\ord{\partial_x^p}.
\end{itemize}
\end{theorem} 

\begin{proof} 
By \cref{Xlemdiffmed} the spatial discretisation of the evolution~\eqref{Xeqdiffpde} satisfies~\cref{Xeqeqpdediff}.
Consider macroscale solutions of~\eqref{Xeqeqpdediff} for which the differences~\(\delta_j\) are `small':
for such solutions, the Taylor series in~\(\delta_j\) is appropriate.
We expand \(\asinh\) in~\cref{Xeqeqpdediff} but, similarly to the proof of \cref{thmfocwave}, only retain the powers of \(\gamma\), \(\mu_j\) and \(\delta_j\), omitting all other details:
\begin{align*}
\asinh[\cdot]&\sim\asinh\left[\frac{\sqrt{(\mu_j^2-\delta_j^2)}(\gamma\delta_j)}{\sqrt{1+\delta_j(\gamma\delta_j)+(\gamma\delta_j)^2}}\right]
\\&\sim \sum_{m=0}^{\infty}(\mu_j^2-\delta_j^2)^{m+1/2}(\gamma\delta_j)^{2m+1}[1+\delta_j(\gamma\delta_j)+(\gamma\delta_j)^2]^{-m-1/2}
\\&\sim \sum_{m=0}^{\infty}(\mu_j^2-\delta_j^2)^{m+1/2}(\gamma\delta_j)^{2m+1}\sum_{n=0}^{\infty}\delta_j^{0:n}(\gamma\delta_j)^{n:2n} 
\\&\sim \sum_{m,n=0}^{\infty}(\mu_j^2-\delta_j^2)^{m+1/2}\delta_j^{0:n}(\gamma\delta_j)^{2m+n+1:2m+2n+1}\,.
\end{align*}
Thus, since \(\mu_j\sim\delta_j\) the \(\asinh^2\) in~\cref{Xeqeqpdediff} expands as
\begin{equation*}
\asinh^2[\cdot]
\sim \delta_j^2(\gamma\delta_j)^2 \left[ \sum_{m,n=0}^{\infty}\delta_j^{2m:2m+n}(\gamma\delta_j)^{2m+n:2m+2n}\right]^2.
\end{equation*}
First establish the theorem's properties for~\eqref{Xeqeqpdediffa}. 
For a given power~\(\delta_j^p\), every~\(\gamma\) exponent is~\(\leq p\).
Hence, truncating the analysis to errors~\ord{\gamma^p} ensures that the terms are all complete and correct for every power of~\(\delta_j\) less than~\(p\); that is, the error is~\ord{\delta_j^p}.
Since \(\partial_x\sim\delta_j\) for smooth fields, the consistency error in such analysis is~\ord{\partial_x^p}.

The highest power of~\(\gamma\) retained is \(p-1\) which could appear in the above sum as \(\gamma^{2+2(2m+n)}\)\,.
For such a power of \(\gamma\), the highest power of the difference is \(\delta_j^{2+2+2(2m+n)+2(2m+n)}=\delta_j^{2(p-1)}\)\,, which involves shifts of \((E_j^{\pm 1/2})^{2(p-1)}=E_j^{\pm(p-1)}\)\,.
That is, a construction to errors~\ord{\gamma^p} results in a macroscale discretisation with stencil width~\((2p-1)\).

Second, establish the theorem's properties for~\eqref{Xeqeqpdediffb}.
Denote the the right-hand side operator of~\eqref{Xeqeqpdediffa} as \(\cR=\frac{4}{H^2}\asinh^2(\cdots)\) so that~\eqref{Xeqeqpdediffb} is \(\partial_t U_j=\cK_0U_j+\cK\{\cR\}U_j\)\,, and recall that \(\cK(\cR)=\sum_{k=1}^\infty\cK_k \cR^k\)\,.
From the above analysis we have 
\begin{equation*}
\cR^k\sim\asinh^{2k}[\cdot]
\sim \delta_j^{2k}(\gamma\delta_j)^{2k} \left[ \sum_{m,n=0}^{\infty}\delta_j^{2m:2m+n}(\gamma\delta_j)^{2m+n:2m+2n}\right]^{2k}.
\end{equation*}
As before, for a given power~\(\delta_j^p\), every~\(\gamma\) exponent is~\(\leq p\) so truncating the  to errors~\ord{\gamma^p} ensures the error is~\ord{\delta_j^p} and the consistency error is~\(\ord{\partial_x^p}\).

When the highest power of~\(\gamma\) retained is~\(p-1\), then this power could appear in~\cK\ as \(\gamma^{2k+2k(2m+n)}\)\,, corresponding to the highest power of the difference operator \(\delta_j^{2k+2k+2k(2m+n)+2k(2m+n)}=\delta_j^{2(p-1)}\)\,.
As before, this involves shifts of \((E_j^{\pm 1/2})^{2(p-1)}=E_j^{\pm(p-1)}\)\,, and so we conclude that 
 a construction to errors~\ord{\gamma^p} results in a macroscale discretisation with stencil width~\((2p-1)\).
\end{proof}

\section{Application to accurate numerical homogenisation}
\label{sechhd}

Consider a general heterogeneous diffusion \pde, for field~\(u(t,x)\) in some spatial domain~\XX,
\begin{equation}
\D tu=\D x{}\left[\kappa(x)\D xu\right],
\label{eqhetdiff}
\end{equation}
where the diffusivity~\(\kappa(x)\) is \(d\)-periodic in~\(x\).
The microscale length~\(d\) is fixed. 
We do \emph{not} invoke the limit \(d\to0\)\,: \(d\)~is some constant that happens to be relatively small compared to the size~\fL\ of the the domain of interest.
Our approach and results here apply to the physically relevant case of a finite scale separation ratio~\(d/\fL\).
The results are \emph{not} restricted to the \text{mathematical limit \(d/\fL\to0\)\,.}

After some general considerations, we focus on the specific example case when the diffusivity is \(\kappa=1/(1+a\cos kx)\) for some amplitude~\(a\) of the heterogeneity.
For \(d\)-periodic heterogeneity, here the wavenumber of the microscale \text{is \(k:=2\pi/d\).}

\begin{figure}
\centering
\caption{\label{figensemble}`cylindrical' domain of the embedding \pde~\cref{eqembeddiff} for field \(\fu(t,\fx,z)\).  
Obtain solutions of the heterogeneous diffusion \pde~\cref{eqhetdiff} on the blue line as \(u_\phi(t,x)=\fu(t,x,x+\phi)\) for every constant phase~\(\phi\).}
\setlength{\unitlength}{0.01\textwidth}
\begin{picture}(92,19)
\put(0,5){\vector(1,0){92}}
\put(40,1){\(x,\fx\)}
\put(5,0){\vector(0,1){18}}
\put(2,16){\(z\)}
\put(2,12){\(d\)}\put(4,13){\line(1,0){2}}
\put(2,5){\(0\)}
\put(30,15){domain \(\XX\times[0,d)\)}
\put(5,5){\framebox(84,8){}}
{\color{blue}
\multiput(9,5)(8,0){10}{
  \put(0,0){\line(1,1){8}}
  \multiput(0,0)(0,2)4{\line(0,1){1}}
  }
  \put(9,13){\line(-1,-1){4}}
  \put(2,9){\(\phi\)}
}
\put(30,8){\(\fu (t,\fx,z)\)}
\end{picture}
\end{figure}
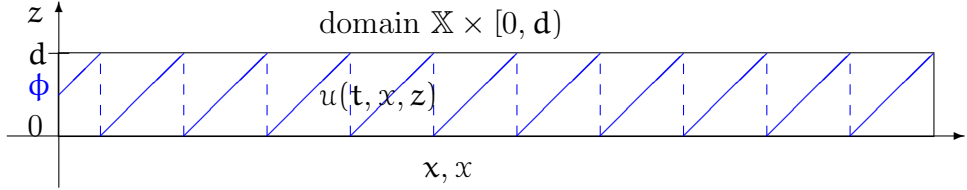

To make rigorous progress we embed the general heterogeneous diffusion \pde~\cref{eqhetdiff} in the family of all phase shifts of the \(d\)-periodic diffusion \cite[e.g.,][]{Roberts2013a, Roberts2016a}.
Consider the following \pde\ for a field \(\fu(t,\fx,z)\) in the `cylindrical' domain~\(\XX\times [0,d)\) (\cref{figensemble}):
\begin{align}
\partial_t\fu
&=(\partial_\fx+\partial_z)\big[\kappa(z)(\fu_\fx+\fu_z)\big] 
\nonumber\\&
=\partial_z[\kappa\partial_z]\fu 
+[2\kappa\partial_z+\kappa']\fu_\fx
+\kappa\fu_{\fx\fx}\,,
\label{eqembeddiff}
\end{align}
with a boundary condition of \(d\)-periodicity in~\(z\).
For every field~\(\fu(t,\fx,z)\) that satisfies \pde~\eqref{eqembeddiff}, define the field~\(u_\phi(t,x):=\fu(t,x,x+\phi)\) where the \(z\)-argument, \(x+\phi\), is taken modulo~\(d\) as indicated in \cref{figensemble}.
Then from the \pde~\cref{eqembeddiff}, the field~\(u_\phi\) satisfies the heterogeneous diffusion \pde\ \(\partial_tu_\phi=\partial_x[\kappa(x+\phi)\partial_xu_\phi]\)\,.
Hence solutions~\(\fu(t,\fx,z)\) of the embedding \pde~\cref{eqembeddiff} provide solutions of the original heterogeneous \pde~\cref{eqhetdiff} for every phase shift~\(\phi\) of the heterogeneity.
For the specific phase shift \(\phi=0\), the field~\(u_0(t,x)\) solves the \text{specific original \pde~\cref{eqhetdiff}.}

The crucial property of the embedding \pde~\cref{eqembeddiff} is that the \pde\ is homogeneous in the longitudinal coordinate~\(\fx\).
The heterogeneity in~\cref{eqembeddiff} appears only in the cross-section coordinate~\(z\).
Consequently, theory and techniques developed for the case of  systems homogeneous in~\(x\) apply to the \(\fx\)-dependence of the embedding \pde~\cref{eqembeddiff}, and thence to the heterogeneous diffusion~\cref{eqhetdiff}.
This phase-shift embedding also extends to nonlinear systems \cite[e.g., \S3.3][]{Roberts2013a}, but here we confine attention to linear homogenisation.

Specifically, \cref{Xthmfocdiff} indicates that a holistic discretisation constructed via inter-element coupling controlled by edge conditions~\eqref{eqsiecc}, and constructed to errors~\ord{\gamma^p}, is a spatially discrete scheme with a stencil width of~\((2p-1)\) that is consistent to the \emph{macroscale} dynamics of the heterogeneous diffusion~\cref{eqhetdiff} to errors~\ord{\partial_x^p}.
The theorem only ``indicates" because in order to apply, the theorem needs to address fields~\(u(t,x)\) in some Hilbert space of \text{the \(z\)-dependence.}

Now consider the specific case of homogenising the specific heterogeneous diffusion \(\kappa(x):=1/(1+a\cos kx)\)\,.
With inter-element edge conditions~\eqref{eqsiecc} and tuning \(\theta=0\)\,, \cref{sechdhdc} lists code that constructs the holistic discretisation, the numerical homogenisation, of the heterogeneous diffusion~\cref{eqhetdiff} in this case.
The code implicitly assumes~\(d\) divides evenly into the element length~\(H\) as otherwise various coded means are not correct.
For simplicity, we seek the results as a power series in both the coupling~\(\gamma\) and the amplitude~\(a\) \text{of the heterogeneity.}

In the ensemble of \cref{figensemble}, the flux in the \(\fx\)-direction (as distinct from the flux in the \(x\)-direction) from the ensemble \pde~\cref{eqembeddiff} is \(f=-2\kappa\fu_z-\kappa'\fu-\kappa\fu_{\fx}\)\,.
Since \(\fu\)~satisfies the edge condition~\cref{eqiecca}, then so does its derivative in the transverse direction~\(\fu_z\).
Hence, the flux edge condition~\cref{eqieccb}, upon dividing by~\(-\kappa(z)\), is satisfied by the derivative~\(\fu_\fx\).
This is the second of the two edge conditions \text{coded in \cref{sechdhdc}.}

\begin{subequations}\label{eqshetdiff}%
Executing the code of \cref{sechdhdc} finds that, to low-orders, the subgrid field 
\begin{align}
\fu_j&=U_j
+\gamma\left\{\xi-\tfrac12 +\frac a{kH}\sin kz \right\}\mu_j\delta_j U_j
\quad{}
+\gamma^2\left\{ \vphantom{\frac12}
(\tfrac1{12}-\tfrac12\xi+\tfrac12\xi^2)\delta_j^2
\right.\nonumber\\&\quad\left.{}
+(\tfrac18-\tfrac14\xi)\mu_j\delta_j^3
+(\tfrac1{48}-\tfrac18\xi+\tfrac18\xi^2)\delta_j^4
+\frac a{k^2H^2}\cos kz\,(\delta_j^2+\tfrac14\delta_j^4)
\right.\nonumber\\&\quad\left.{}
-\frac{a}{kH}\sin kz\,\left[
(\tfrac12-\xi)\delta_j^2
+\tfrac14\mu_j\delta_j^3
+(\tfrac18-\tfrac14\xi)\delta_j^4\right]
\right\}U_j
+\Ord{\gamma^3,a^3}
\label{equjhd}
\end{align}
This shows that the sub-element field has smooth macroscale structures in space through its \(\xi\)-dependence, structures that are modified by the heterogeneity via the microscale diffusivity-variations represented by terms in \(\sin kz\) and \(\cos kz\).
Since their coefficients involve divisions by wavenumber~\(k\), which are proportional to multiplication by the small microscale periodicity~\(d\), these trigonometric modifications to the field are relatively small.
Higher-order terms involve harmonics of these \text{trigonometric functions.}

Executing the code of \cref{sechdhdc} to higher-order errors we find that the corresponding evolution modifies the homogeneous case~\cref{eqgopdiff} to
\begin{align}
H^2\partial_tU_j&= \cdots\cref{eqgopdiff}\cdots
+\frac{a^2}{2k^2H^2}(\gamma^4-\gamma^5\delta_j^2)
\mu_j^4\delta_j^4U_j
+\Ord{\gamma^6,a^7},
\label{eqgophetdiff}
\end{align}
where the microscale heterogeneity only affects the macroscale evolution through  the magnitude~\(a\) and the wavenumber~\(k\) of the heterogeneity.
Alternatively, executing the code to errors~\Ord{\gamma^7,a^3}, and evaluated at full coupling \(\gamma=1\), we compute that this evolution has equivalent \pde
\begin{align}
\D tU&=\DD xU
+\frac{a^2}{2k^2}\Dn x4U
-\frac{2a^2}{k^4}\Dn x6U
+\Ord{a^2H^2d^4\partial_x^8U}.
\label{eqedehetdiff}
\end{align}
\end{subequations}
The leading-order homogenisation, the diffusion term~\(\DD xU\), is exactly the well-known correct harmonic mean of the diffusivity~\(\kappa(x)\). 
The higher-order terms,~\(\Dn x4U\), being divided by powers of~\(k=2\pi/d\), vanish in the usual theoretical limit of \(d\to0\)\,.
However, here our analysis is rigorous for finite~\(d\) and so these fourth and higher order terms quantify effects due to the physical finite scale separation of the macroscale from a finite sized microscale.   
These higher-order derivatives depend upon~\(a^2\) (recall that \(a\)~is the magnitude of the heterogeneity), as appropriate by symmetry in~\(a\).

Consequently, the holistic discretisation~\eqref{eqgophetdiff} is an accurate, analytically learnt, numerical homogenisation for the heterogeneous diffusion~\eqref{eqhetdiff}.

\paragraph{Numerical homogenisation for waves}
Wave propagation (\cref{secsmwpde}) through heterogeneous media can be analysed almost identically to the heterogeneous diffusion of this section.
The resultant numerical homogenisation would be~\eqref{eqgophetdiff} but with \(\partial_t^2U_j\) on the left-hand side.
Its equivalent \pde\ would be~\eqref{eqedehetdiff} but with \(\DD tU\) on the left.
Consequently, the numerical homogenisation would predict, through the fourth- and sixth-order differences\slash derivatives in~\eqref{eqshetdiff}, a wave speed dependence upon wavelength (wave dispersion) caused by the microscale heterogeneity at \text{finite scale separation.}

\paragraph{Spatial boundaries}
The homogenisation~\eqref{eqgophetdiff} is constructed with periodic diffusivity and is independent of the phase~\(\phi\) of the microscale heterogeneity.
Macroscale effects of the phase~\(\phi\) only arise via correctly determined influences of the boundary conditions on the macroscale domain.
The development of homogenisation for general boundary conditions is the subject of ongoing research. 
For spatial discretisations, such as~\eqref{eqgophetdiff}, one could develop correct discretisations near a boundary by adapting the arguments of \cite{Roberts01b, MacKenzie03}.

\section{Conclusion}

Fine-scale heterogeneity or nonlinearities both complicate the derivation of macroscale discretisation of \pde{}s, with many common methods failing to accurately account for the effects at the macroscale of subgrid physics. 
Holistic discretisations systematically construct macroscale closures which can accurately account for subgrid structures. 
In particular, for an accurate homogenisation of a \pde{}, the discretisation must maintain symmetries of \pde{}.
Here we developed a new holistic discretisation which is guaranteed to maintain the self-adjointness of a \pde{} through carefully crafted inter-element coupling conditions, thus ensuring that the general spectral structure of the \pde\  is captured by the discretisation.
For both generic 1D waves and diffusion, we show that the homogenisation is consistent with the original \pde{}.

In this article we focus on 1D systems, but simulations of 2D systems indicate that analogous discretisations are possible in 2D. 
Such 2D systems will be explored and theory developed in future research.

\paragraph{Acknowledgement}
This research was funded by the Australian Research Council under grants DP150102385 and DP200103097.  
We thank Peter Hochs for his comments on earlier versions of this article.

\appendix

\section{High-order consistency for holistic discretisation of advection-diffusion}
\label{sechochd}

This is script \verb|diffAdvecHolistic.tex|\quad
\input{diffAdvecHolistic}

\subsection{Convert evolution to equivalent PDE}
\label{secceePDE}
All algorithms invoke this conversion script \verb|convert2EquivalentDE.tex|
\input{convert2EquivalentDE.tex}

\section{High-order consistency for holistic discretisation of simple wave PDE}
\label{sechochwave1}
This is script \verb|wave1sm.tex|\quad
\input{wave1sm}

\section{Holistic discretisation for heterogeneous diffusion and its consistency}
\label{sechdhdc}
This is script \verb|hetDiffHolistic.tex|\quad
\input{hetDiffHolistic}

%

\end{document}

%% file: diffAdvecHolistic.tex
Learns the holistic discretisation of advection-diffusion on
elements with self-adjoint preserving coupling.  Find on
slow manifold both \(u\) and \(dU/dt\) linear in advection
\(c\), although the code must truncate in \(c\) in order to
converge.  All code is written in the computer algebra
package {Reduce}.\footnote{Reduce
[\url{http://reduce-algebra.com/}] is  a free, fast, general
purpose, computer algebra system.}
\begin{reduce}
on div; off allfac; on revpri; 
factor gamma,hh,c;
\end{reduce}
Subgrid structures are functions of sub-element variable
\(\xi=(x-L_j)/H\)
\begin{reduce}
depend xi,x;  let df(xi,x)=>1/hh;
\end{reduce}
Operator \verb|linv| solves \(u''=\rhs\) such that
\(u(1)=u(0)\) and mean zero.
\begin{reduce}
operator linv; linear linv;
let { linv(xi^~~p,xi)=>(xi^(p+2)-xi+1/2-1/(p+3))/(p+1)/(p+2)
    , linv(1,xi)=>(xi^2-xi+1/2-1/3)/2 };
\end{reduce}
Operator to compute mean over an element
\begin{reduce}
operator mean; linear mean;
let { mean(xi^~~p,xi)=>1/(p+1)
    , mean(1,xi)=>1 };
\end{reduce}
    
Parametrise discretisation slow manifold with evolving order
parameters \verb|uu(j)|
\begin{reduce}
operator uu;  depend uu,t; 
let df(uu(~k),t)=>sub(j=k,gj);
\end{reduce}
Initial approximation in \(j\)th element
\begin{reduce}
uj:=uu(j);  gj:=0;
\end{reduce}

Here specify required orders of errors in the result.  Intermediate working needs \(c\) truncated to some order
\begin{reduce}
let { gamma^3=>0, c^3=>0}; 
\end{reduce}
Deep iterative refinement learns emergent slow manifold
\begin{reduce}
R:= xi=1;  L:= xi=0;
for iter:=1:99 do begin
\end{reduce}
   Compute residuals for selected one of possible PDEs, the
   two coupling conditions, and the definition of the
   order-parameter macroscale variable.
\begin{reduce}
   pde:= -df(uj,t) + (if 1
     then -c*df(uj,x)+df(uj,x,2)
     else (sub(xi=xi+d,uj)-2*uj+sub(xi=xi-d,uj))/(d*hh)^2 );
   ucc:= -(1-gamma/2)*(sub(R,uj)-sub(L,uj))
     +gamma/2*(sub({L,j=j+1},uj)-sub({R,j=j-1},uj))
     +gamma*theta/2*(sub(R,uj)+sub(L,uj))
     -gamma*theta/2*(sub({L,j=j+1},uj)+sub({R,j=j-1},uj));
   ux:=df(uj,x);
   udc:= -(1-gamma/2)*(sub(R,ux)-sub(L,ux))
     +gamma/2*(sub({L,j=j+1},ux)-sub({R,j=j-1},ux))
     -gamma*theta/2*(sub(R,ux)+sub(L,ux))
     +gamma*theta/2*(sub({L,j=j+1},ux)+sub({R,j=j-1},ux));
   amp:=mean(uj,xi)-uu(j);
\end{reduce}
   Trace write lengths of residuals.
\begin{reduce}
   write lengthress:=map(length(~a),{pde,ucc,udc,amp});
\end{reduce}
   Update approximations from the current residuals
\begin{reduce}
   gj:=gj+(gd:=udc/hh-mean(pde,xi));
   uj:=uj-hh^2*linv(pde-gd,xi)+(xi-1/2)*ucc;
\end{reduce}
   Exit loop when residuals are zero to specified order of
   error
\begin{reduce}
   if {pde,ucc,udc,amp}={0,0,0,0} 
   then write "Success: ",iter:=100000+iter;
end;
if {pde,ucc,udc,amp}neq{0,0,0,0} then rederr("iteration fail");
\end{reduce}

Compute and report equivalent \pde\ of discrete
\begin{reduce}
in_tex "../convert2EquivalentDE.tex"$ 
\end{reduce}
Optionally check some identities, and exit script.
\begin{reduce}
in "diffTestProof.txt"$ 
end;
\end{reduce}

%% file: convert2EquivalentDE.tex
\begin{reduce}
write "Convert evolution dUjdt=gj to equivalent PDE for
U(x,t) via operator form.  AJR, from a long time ago.";
\end{reduce}

Convert to central difference operator form
\begin{reduce}
rules:={ mu^2=>1+delta^2/4, uu(j)=>1
    , uu(j+~p)=>(1+sign(p)*mu*delta+delta^2/2)^abs(p)}$
gop:=(gj where rules);
uop:=(uj where rules)$
\end{reduce}

Convert to equivalent \pde\ using Taylor expansion
\begin{reduce}
remfac gamma; factor df;
let hh^10=>0;
depend uu,x;
rules:={uu(j)=>uu, uu(j+~p)=>uu+(for n:=1:10 sum 
               df(uu,x,n)*(hh*p)^n/factorial(n)) }$
duujdt:=(gj where rules);
duujdt1:=sub(gamma=1,duujdt);
\end{reduce}

\begin{reduce}
end;
\end{reduce}

%% file: wave1sm.tex
Holistic discretisation of modified uni-directional
wave PDE, with speed c, on elements with `self-adjoint'
coupling.  The sub-element field is independent of speed c. For
error~\Ord{\gamma^p} gives consistency~\Ord{H^{p-1}}.  For linear
modifications measured by alpha, the slow manifold is linear
in alpha, but must truncate in alpha for iteration to
terminate!   

\begin{reduce}
on div; off allfac; on revpri; 
factor gamma,hh,c,alpha,c0,c2,c3,c4;
\end{reduce}
Subgrid structures are functions of sub-element variable
\(\xi=(x-L_j)/H\)
\begin{reduce}
depend xi,x;  let df(xi,x)=>1/hh;
\end{reduce}
Operator \verb|linv| solves \(u'=\rhs\) such that mean is zero.
\begin{reduce}
operator linv; linear linv;
let { linv(xi^~~p,xi)=>(xi^(p+1)-1/(p+2))/(p+1)
    , linv(1,xi)=>(xi-1/2) };
\end{reduce}
Operator to compute mean over an element
\begin{reduce}
operator mean; linear mean;
let { mean(xi^~~p,xi)=>1/(p+1)
    , mean(1,xi)=>1 };
\end{reduce}
    
Parametrise discretisation slow manifold with evolving order
parameters \verb|uu(j)|
\begin{reduce}
operator uu;  depend uu,t; 
let df(uu(~k),t)=>sub(j=k,gj);
\end{reduce}
Initial approximation in \(j\)th element
\begin{reduce}
uj:=uu(j);  gj:=0;
\end{reduce}

Here specify required orders of errors in the result. 
\begin{reduce}
let { gamma^6=>0, alpha=>0 };
\end{reduce}
Deep iterative refinement learns emergent slow manifold
\begin{reduce}
R:= xi=1;  L:= xi=0;
for iter:=1:99 do begin
\end{reduce}
   Compute residuals for selected one of possible PDEs, the
   coupling condition, and the definition of the
   order-parameter macroscale variable.
\begin{reduce}
   pde:= -df(uj,t) + (if 1
     then -c*df(uj,x)
       +alpha*(c0*uj+c2*df(uj,x,2)+c3*df(uj,x,3)+c4*df(uj,x,4))
     else -c*(sub(xi=xi+d,uj)-sub(xi=xi-d,uj))/(2*d*hh) );
   fx:=c*uj; the 'flux'
   ucc:=(1+theta)/2*( -sub(L,fx)
           +(1-gamma)*sub(R,fx) +gamma*sub({R,j=j-1},fx) )
       -(1-theta)/2*(-sub(R,fx)
           +(1-gamma)*sub(L,fx) +gamma*sub({L,j=j+1},fx) );           
   amp:=mean(uj,xi)-uu(j);
\end{reduce}
   Trace write lengths of residuals.
\begin{reduce}
   write lengthress:=map(length(~a),{pde,ucc,amp});
\end{reduce}
   Update approximations from the current residuals
\begin{reduce}
   gj:=gj+(gd:=ucc/hh+0*mean(pde,xi));
   uj:=uj+hh/c*linv(pde-gd,xi);
\end{reduce}
   Exit loop when residuals are zero to specified order of
   error
\begin{reduce}
   if {pde,ucc,amp}={0,0,0} 
   then write "Success: ",iter:=100000+iter;
end;
if {pde,ucc,amp}neq{0,0,0} then rederr("iteration fail");
\end{reduce}

Compute and report equivalent \pde, then exit script.
\begin{reduce}
in_tex "../convert2EquivalentDE.tex"$
end;
\end{reduce}

%% file: hetDiffHolistic.tex
Holistic discretisation of heterogeneous diffusion on
elements with self-adjoint preserving coupling.  Here seek
discretisation as series in \(a\), the amplitude of the
heterogeneity.  
\begin{reduce}
on div; off allfac; on revpri; 
factor gamma,hh,a,k;
\end{reduce}
Subgrid structures are functions of sub-element variable
\(\xi=(x-L_j)/H\)
\begin{reduce}
depend xi,x;  let df(xi,x)=>1/hh;
depend xi,xz; depend z,xz;
\end{reduce}
Operator \verb|linv| solves \(u''=\rhs\) such that
\(u(1)=u(0)\), periodic in~\(z\), and mean zero.
\begin{reduce}
operator linv; linear linv;
let { linv(xi^~~p,xz)=>hh^2*(xi^(p+2)-xi+1/2-1/(p+3))/(p+1)/(p+2)
    , linv(1,xz)=>hh^2*(xi^2-xi+1/2-1/3)/2 
    , linv(sin(~~q*z),xz)=>-1/(q)^2*sin(q*z)
    , linv(cos(~~q*z),xz)=>-1/(q)^2*cos(q*z)
    , linv(sin(~~q*z)*xi^~~p,xz)=>-(xi^p-xi)/(q)^2*sin(q*z)
    , linv(cos(~~q*z)*xi^~~p,xz)=>-(xi^p-xi)/(q)^2*cos(q*z)
    };
\end{reduce}
Operator to compute mean over an element, these assume trig
functions fit in period
\begin{reduce}
operator mean; linear mean;
let { mean(1,xz)=>1 , mean(xi^~~p,xz)=>1/(p+1)
    , mean(cos(~~q*z),xz)=>0 , mean(xi^~~p*cos(~~q*z),xz)=>0 
    , mean(sin(~~q*z),xz)=>0 , mean(xi^~~p*sin(~~q*z),xz)=>0 
    };
\end{reduce}
    
Parametrise discretisation slow manifold with evolving order
parameters \verb|uu(j)|
\begin{reduce}
operator uu;  depend uu,t; 
let df(uu(~k),t)=>sub(j=k,gj);
\end{reduce}
Initial approximation in jth element
\begin{reduce}
uj:=uu(j);  gj:=0;
\end{reduce}

Iterative refinement to specified order of error, and using
Taylor series of~\(\kappa(x)\) in coefficient~\(a\)
\begin{reduce}
let { gamma^7=>0, a^3=>0 };
kappa:=for n:=0:deg((1+a)^9,a) sum (-a*cos(k*z))^n;
for iter:=1:99 do begin 
\end{reduce}
   Compute residuals for the PDE, the two coupling
   conditions, and the definition of the order-parameter
   macroscale variable.
\begin{reduce}
   flux:=kappa*(df(uj,x)+df(uj,z));
   pde:=trigsimp( -df(uj,t)+df(flux,x)+df(flux,z) 
       ,combine);
   ucc:=(1-gamma/2)*(-sub(xi=+1,uj)+sub(xi=0,uj))
   +gamma/2*(sub({xi=0,j=j+1},uj)-sub({xi=1,j=j-1},uj));
   ux:=df(uj,x);
   udc:=(1-gamma/2)*(-sub(xi=+1,ux)+sub(xi=0,ux))
   +gamma/2*(sub({xi=0,j=j+1},ux)-sub({xi=1,j=j-1},ux));
   amp:=mean(uj,xz)-uu(j);
\end{reduce}
   Trace write lengths of residuals.
\begin{reduce}
   write lengthress:=map(length(~a),{pde,ucc,udc,amp});
\end{reduce}
   Update approximations from the current residuals
\begin{reduce}
   gj:=gj+(gd:=mean(udc,xz)/hh+mean(pde,xz));
   uj:=uj-linv(pde-gd,xz)+(xi-1/2)*ucc;
\end{reduce}
   Exit loop when residuals are zero to specified order of
   error
\begin{reduce}
   if {pde,ucc,udc,amp}={0,0,0,0} 
   then write "Success: ",iter:=100000+iter;
end;
if {pde,ucc,udc,amp}neq{0,0,0,0} then rederr("iteration fail");
\end{reduce}

Compute and report equivalent \pde, then exit script.
\begin{reduce}
in_tex "../convert2EquivalentDE.tex"$
end;
\end{reduce}